\documentclass[final]{siamltex}
\usepackage{amsfonts,amsmath,mathrsfs,bm}
\usepackage{amssymb}
\usepackage[notref,notcite]{showkeys}
\usepackage{hyperref}
\usepackage{graphicx}
\usepackage{wrapfig}
\usepackage{subfig}
\usepackage{float}
\usepackage{bbm}
\usepackage{algorithmic, algorithm}
\usepackage{gensymb,color,soul}
\usepackage{tikz}
\usepackage{tikz-3dplot}

\newcommand{\R}{\mathbb{R}}
\newcommand{\C}{\mathbb{C}}
\newcommand{\E}{\mathbb{E}}

\newtheorem{remark}[theorem]{Remark}

\DeclareMathOperator{\argmax}{\rm arg\, max}

\title{Bayesian experimental design for linear elasticity}

\author{
S. Eberle-Blick\footnotemark[2]
\and N. Hyv\"onen\footnotemark[3]
}

\begin{document}
\maketitle

\renewcommand{\thefootnote}{\fnsymbol{footnote}}
\footnotetext[2]{Goethe-University Frankfurt, Institute of Mathematics, Frankfurt am Main, Germany (eberle@math.uni-frankfurt.de). The work of SE-B was supported by the German Research Foundation (DFG): project number 499303971}
\footnotetext[3]{Aalto University, Department of Mathematics and Systems Analysis, P.O.~Box 11100, FI-00076 Aalto, Finland (nuutti.hyvonen@aalto.fi). The work of NH was supported by the Academy of Finland (decision 348503, 353081).}

\begin{abstract}
  This work considers Bayesian experimental design for the inverse boundary value problem of linear elasticity in a two-dimensional setting. The aim is to optimize the positions of compactly supported pressure activations on the boundary of the examined body in order to maximize the value of the resulting boundary deformations as data for  the inverse problem of reconstructing the Lam\'e parameters inside the object. We resort to a linearized measurement model and adopt the framework of Bayesian experimental design, under the assumption that the prior and measurement noise distributions are mutually independent Gaussians. This enables the use of the standard Bayesian A-optimality criterion for deducing optimal positions for the pressure activations. The (second) derivatives of the boundary measurements with respect to the Lam\'e parameters and the positions of the boundary pressure activations are deduced to allow minimizing the corresponding objective function,~i.e.,~the trace of the covariance matrix of the posterior distribution, by a gradient-based optimization algorithm. Two-dimensional numerical experiments are performed to demonstrate the functionality of our approach.
\end{abstract}

\renewcommand{\thefootnote}{\arabic{footnote}}



\pagestyle{myheadings}
\thispagestyle{plain}
\markboth{S. EBERLE-BLICK AND N. HYV\"ONEN}{EXPERIMENTAL DESIGN FOR LINEAR ELASTICITY}

\section{Introduction}\label{intro}
Nondestructive testing based on mechanical probing of a physical body can be utilized in engineering, geosciences and medical imaging \cite{Ammari15,Doyley12,Oberai04}. Under suitable assumptions, such testing can be mathematically formulated as a quest for information on the Lam\'e parameters inside the investigated object in the framework of linear elasticity. We refer to~\cite{Barbone04, Beretta14a, Beretta14b, Carstea18, Eskin02, Ikehata90, Ikehata06, Ikehata99, Imanuvilov11, Nakamura99,Nakamura93,Nakamura94,Nakamura95} for theoretical results and to~\cite{Andrieux99,Doubova20, Eberle21c, Eberle21a,Eberle22, Eberle23, Ferrier19,Hubmer18,Jadamba08,Lin17,Marin02,Marin05,Oberai03,Seidl19,Seidl20,Steinhorst12} for reconstruction methods related to the inverse problem of linear elasticity. In this work, we acknowledge that any practical measurement setting related to nondestructive testing in the framework of linear elasticity allows only a finite number of boundary pressure activations. Our aim is to choose the activation positions so that the value of the measurements on the resulting boundary deformations of the examined body is maximized in the inverse problem of reconstructing the Lam\'e parameters. Many aspects of our work are motivated by the experimental setup and results in \cite{Eberle21b}. 

A Bayesian optimal design $p^*$ maximizes over the set of admissible designs $\mathcal{P}$ the expectation of the utility function $\E_{u,y}[U( p ; u,y)]$, with $y \in \mathcal{Y}$ being the data and $ u\in \mathcal{U}$ the unknown in the studied inverse problem~\cite{chaloner1995bayesian}. In our setting, the design parameter $p$ determines the positions of the employed pressure activations on the boundary of the imaged object, the unknown $u$ corresponds to the Lam\'e parameters, and $y$ carries the data on the measured boundary deformations. More concretely,
\begin{equation}
	\label{eq:OED_task}
	p^* 
	 =  \underset{p\in \mathcal{P}}{\argmax} \int_{\mathcal{Y}} \int_\mathcal{U} U(p; u, y) \pi(u \, | \, p, y) \pi(y \, | \, p) \, {\rm d} u \, {\rm d}y,
\end{equation}
where $\pi(u \, | \, p, y)$ and $\pi(y \, | \, p)$ are the posterior distribution for the unknown and the marginalized distribution of the data, respectively, for the design $p$. We consider a standard choice for the utility $U$, namely the \emph{negative quadratic loss function} that measures the distance from $u$ to the posterior mean.

The mere evaluation of the double-integral on the right-hand side of \eqref{eq:OED_task} can be intractably expensive if the dimension of the data space $\mathcal{Y}$ or/and the parameter space $\mathcal{U}$ is high, which is often the case for inverse boundary value problems. However, if the relation between $u$ and $y$ is (assumed to be) linear and the prior for $u$ and the additive measurement noise are mutually independent Gaussians, the double-integral essentially reduces to the trace of the posterior covariance when the aforementioned quadratic loss plays the role of the utility function; a minimizer of this simplified target function is called a Bayesian A-optimal design~\cite{alexanderian2016bayesian,chaloner1995bayesian}. Motivated by this observation, we restrict our attention to the linearized inverse problem of linear elasticity, that is, we replace the nonlinear forward map that sends the Lam\'e parameters to the boundary operator mapping boundary activations to the resulting deformations by its linearization around a background Lam\'e parameter pair.  After discretization, this enables writing the A-optimality target function explicitly with the help of the linearized forward map and the prior and noise covariance matrices. In order to apply a gradient-based minimization algorithm to finding an A-optimal design for the boundary pressure activations, we also introduce the derivative of the $u$-linearized forward map with respect to the positions of the boundary activations parameterized by $p$.

The main contribution of this work is introducing and testing optimization methods for searching A-optimal positions of the boundary pressure activations for the inverse problem of (linearized) linear elasticity. In particular, we are not aware of previous works on applying Bayesian experimental design to the considered setting, although \cite{Etling18} also considers experimental design for linear elasticity but from a different standpoint. Our algorithms are typically able to significantly reduce the value of the A-optimality target function, but they are not guaranteed to locate the globally optimal pressure activation pattern because the target function is expected to suffer from multiple local minima --- especially if many activations are involved in the minimization process. We test a simple heuristic for mitigating this problem. Note that the optimization algorithms can be run offline,~i.e.,~prior to performing any measurements, since in a linear(ized) Gaussian setting the posterior covariance matrix, which defines the A-optimality target, does not depend on the measured data. Moreover, although the optimization of the locations where the boundary deformation is measured could be tackled in exactly the same way due to the symmetry of the underlying partial differential equation, we restrict our attention solely to the pressure activations. A closely related approach for optimizing electrode positions in electrical impedance tomography was studied in~\cite{hyvonen2014eit}.

Our approach is built on a linearization and discretization of the studied inverse problem and the underlying Bayesian {\em optimal experimental design} (OED) problem. As mentioned above, the motivation for the linearization is to allow explicit integration of \eqref{eq:OED_task}. However, there exist approaches to tackling Bayesian OED without such a simplifying assumption; see,~e.g.,~\cite{alexanderian2016fast,beck2018fast,huan2010accelerated,huan2013simulation,long2013fast,wu2021fast,wu2020fast}. Moreover, one can also aim to avoid discretizing the problem setting before employing OED; see,~e.g.,~the series of papers on Bayesian OED in the framework of infinite-dimensional inverse problems~\cite{alexanderian2016bayesian,alexanderian2014optimal,alexanderian2016fast,alexanderian2021optimal}. The stability of the expected utility under approximations, such as linearization and discretization, in Bayesian OED has recently been investigated in~\cite{duong2022stability}. For general reviews on the topic of Bayesian OED, we refer to~\cite{alexanderian2021optimal_review,chaloner1995bayesian, rainforth2023modern,ryan2016review}.

This text is organized as follows. The forward model of linear elasticity, as well as its Fr\'echet derivatives with respect to the Lam\'e parameters and the positions of the boundary activations, is described in Section~\ref{sec:forward}. Section~\ref{sec:discretized} discretizes the forward map and its derivatives, and Section~\ref{sec:bayes} introduces the finite-dimensional setting for Bayesian inversion and OED. The implementation of the optimization algorithm is discussed in Section~\ref{sec:implementation}, and the numerical experiments are presented in Section~\ref{sec:numerics}. Finally, Section~\ref{sec:conclusions} lists the concluding remarks.

\section{Forward model and its differentiability}
\label{sec:forward}

This section first describes our model for varying the pressure activation on the boundary of the examined two-dimensional object. Subsequently, the needed (second) Fr\'echet derivatives of the (boundary) deformation field are introduced.

\subsection{Forward model}
\label{sec:real_forward}

Let $\Omega\subset \mathbb{R}^d$, $d=2$,\footnote{Apart from the parametrization for the positions of the pressure activation on $\partial \Omega$, most of the presented analysis would also be valid for $d=3$.} be an open simply-connected domain with a $C^{1,\alpha}$, $\alpha > 0$, boundary, and denote the exterior unit normal of $\partial \Omega$ by $\nu$. The boundary $\partial \Omega = \overline{\Gamma}_{\rm N} \cap  \overline{\Gamma}_{\rm D}$ is decomposed into disjoint open Neumann and Dirichlet parts $\Gamma_{\rm N}$ and $\Gamma_{\rm D}$. Assume further that we investigate the Lam\'e parameter pair $\tau = (\lambda,\mu) \in L^\infty_+(\Omega)^2$ around some background value $\tau_0 = (\lambda_0,\mu_0) \in L^\infty_+(\Omega)^2$, with
$$
L^\infty_+(\Omega) = \big\{ \kappa \in L^\infty(\Omega) \ \big| \ {\rm essinf} \big( {\rm Re} (\kappa) \big) > 0 \big\}
$$
denoting the space of essentially bounded functions with real parts that are strictly positive.

The variational formulation of the standard forward problem of linear elasticity with a square-integrable pressure field $g: \Gamma_{\rm N} \to \R^d$ as the boundary load,~i.e.,
\noindent
\begin{align}
\nabla \cdot \left(\lambda (\nabla\cdot u)I + 2\mu \hat{\nabla} u \right) &=0 \,\,\quad \text{in}\,\,\Omega,\label{mod_direct_1}\\
\left(\lambda (\nabla\cdot u)I + 2\mu \hat{\nabla} u \right) \nu &={g} \quad \text{on}\,\, \Gamma_{\textup N},\\
u&=0 \quad \,\, \text{on}\,\, \Gamma_{\textup D},\label{mod_direct_3}
\end{align}
\noindent
is to find $u \in\mathcal{V}$ such that~\cite{Ammari15}
\begin{align}
  \label{var_form}
B_{\tau}(u,v) =\int_{\Gamma_{\textup N}}g \cdot v \,{\rm d}s \quad \text{ for all } v\in \mathcal{V}.
\end{align}
The bilinear form $B_{\tau}: \mathcal{V} \times \mathcal{V} \to \C$ and the variational space $\mathcal{V}$ are, respectively, defined by
\begin{equation}
  \label{eq:bilinear}
  B_{\tau}(w,v) = \int_{\Omega} 2 \mu\, \hat{\nabla}w : \hat{\nabla}v  + \lambda \nabla \cdot w \,\nabla\cdot  v\,{\rm d}x
\end{equation}
and 
\[
\mathcal{V}:=\left\{   v\in H^1(\Omega)^d:  v|_{\Gamma_{\textup D}}=0\right\},
\]
with the latter equipped with the norm of $H^1(\Omega)^d$. In what follows, the boundary pressure $g = g_p$ is parametrized by $p\in\mathbb{R}$ that (periodically) defines its position on $\partial \Omega$. To simplify the analysis, we model $g_p$ as a function on the whole of $\partial \Omega$, with the understanding that $g_p|_{\Gamma_{\rm N}}$ defines the actual load in \eqref{var_form}. The solution to \eqref{var_form} corresponding to $g_p \in L^2(\partial \Omega; \R^d)$ is denoted by $u_p \in \mathcal{V}$.
  
It is well known that the bilinear form $B_{\tau}: \mathcal{V} \times \mathcal{V} \to \C$ is continuous and coercive~\cite{Ammari15}: for all $w,v \in \mathcal{V}$,
\begin{align}
  \big | B_{\eta}(w, v) \big| &\leq C \| \eta \|_{L^\infty(\Omega)^2} \| w \|_{H^1(\Omega)^d} \| v \|_{H^1(\Omega)^d}, \label{eq:cont} \\[1mm]
   {\rm Re} \big( B_{\tau}(w, \overline{w}) \big) &\geq c \| w \|_{H^1(\Omega)^d}^2, \label{eq:coer}
\end{align}
where $C = C(\Omega) > 0$ and $c = c(\Omega, \Gamma_{\rm D}, \tau) > 0$. Take note that \eqref{eq:cont} holds for all $\eta \in L^\infty(\Omega)^2$, whereas \eqref{eq:coer} is valid only for $\tau \in L^\infty_+(\Omega)^2$. Moreover, the positive constant $c$ in \eqref{eq:coer} can be chosen to be independent of $\tau \in \mathcal{B}$ for any closed and bounded subset $\mathcal{B} \subset L^\infty_+(\Omega)^2$. In particular, as the right-hand side of \eqref{var_form} obviously defines a continuous linear form on $\mathcal{V}$, it follows from the Lax--Milgram theorem that
\begin{equation}
  \label{eq:bounded}
  \| u_p \|_{H^1(\Omega)^d} \leq C'  \| g_p|_{\Gamma_{\rm N}} \|_{L^2(\Gamma_{\rm N})^d} \leq C' \| g_p \|_{L^2(\partial \Omega)^d},
\end{equation}
where $C' = C'(\Omega, \Gamma_{\rm D}, \tau) > 0$ is independent of $g_p$.

\subsection{Fr\'echet derivatives of the forward operator}
\label{sec:frechet}
For the definitions and analysis of this section, it is essential to recall that $d=2$. Let us explicitly introduce the dependence of $g_p: \partial \Omega \to \R^d$ on its position $p \in \R$ and consider the Fr\'echet differentiability of the map  $\mathbb{R}\ni p \mapsto g_p \in L^2(\partial \Omega)^d$. To this end, let $\partial\Omega$ be parametrized with respect to its arclength as
\begin{align*}
[0,L)\ni s \mapsto \gamma(s)\in\mathbb{R}^d
\end{align*}
and continue $\gamma$ to be an $L$-periodic mapping in $C^{1,\alpha}(\R; \R^d)$. We mildly abuse the notation by denoting the `shape' of the boundary pressure field by $g \in W^{1,\infty}(\partial \Omega)^d$; it is represented in our arclength parametrization as $\tilde{g}=g \circ \gamma: \R \to \mathbb{R}^d$. A family of boundary pressure fields is defined by moving $\tilde{g}$ along the real axis as
\begin{align*}
\tilde{g}_p(s) =\tilde{g}\circ T_{-p}(s) = \tilde{g}\circ T_{s}(-p) = \tilde{g}_{-s}(-p) = \tilde{g}(s-p), \qquad s, p \in \R,
\end{align*}
where $T_q(t) = t+q$ denotes a translation by $q$ on $\R$. The parameter-dependent boundary pressure field, 
\begin{align*}
\R \ni p \mapsto g_p \in  W^{1,\infty}(\partial \Omega)^d \subset L^2(\partial \Omega)^d,
\end{align*}
is then defined as $g_p = \tilde{g}_p \circ {\gamma^{-1}}|_{\partial \Omega}$, where $\gamma$ is treated as a bijective mapping between $[0, L)$ and $\partial \Omega$. Take a note that in what follows we continue to implicitly assume that $g \in W^{1,\infty}(\partial \Omega)^d$,~i.e.,~we require Lipschitz continuity from the shape of the boundary pressure.

\begin{lemma}
  \label{lemma:p_diff}
The mapping $\mathbb{R}\ni p \mapsto g_p \in L^2(\partial \Omega)^d$ is Fr\'echet differentiable. The associated derivative at $p \in \R$ is given by the linear map
  $$
  D_{p}g_p: 
  \left\{
  \begin{array}{l}
  h \mapsto - h \, g_p', \\[1mm]
  \R \to L^2(\partial \Omega)^d,
\end{array}
\right.
$$
where $g_p' = \tilde{g}_p' \circ {\gamma^{-1}}|_{\partial \Omega}$ and $\tilde{g}_p' = \tilde{g}' \circ  T_{-p}$, with $\tilde{g}'$ denoting the weak derivative of $\tilde{g}$.
\end{lemma}

\begin{proof}
  Since by assumption $g \in  W^{1,\infty}(\partial \Omega)^d$ and $\partial \Omega$ is of class $C^{1,\alpha}$, obviously also $\tilde{g} \in  W^{1,\infty}(\R)^d$. Hence, $\tilde{g}$ is Lipschitz continuous (after being modified on a set of zero measure), and thus it is differentiable almost everywhere on $\R$. The same conclusions also hold for the translated version $\tilde{g}_p = \tilde{g} \circ T_{-p}$ due to the smoothness of a translation. In particular,
  $$
  \frac{1}{h} \big( \tilde{g}_{p+h} (s) - \tilde{g}_p(s) + h \, \tilde{g}_p'(s) \big) = - \frac{1}{h} \big( \tilde{g}_p (s) - \tilde{g}_p(s-h) - h \, \tilde{g}_p'(s) \big)  \to 0 \quad \text{as } 0 \not= h \to 0
  $$
  for almost all $s \in [0,L)$. Moreover,
  $$
  \left|\frac{1}{h} \big( \tilde{g}_p (s) - \tilde{g}_p(s-h) - h \, \tilde{g}_p'(s) \big) \right|^2 \leq C < \infty
  $$
  for almost all $s \in \R$ and all $h \in \R$ due to the Lipschitz continuity of $\tilde{g}$ and since $\tilde{g}_p' \in L^\infty(\R)^d$. Hence, it follows from the dominated convergence theorem that
  $$
  \frac{1}{h} \big\| g_{p+h} - g_p +  h \, g_p' \big\|_{L^2(\partial \Omega)} \leq
  \frac{C}{h} \big\| \tilde{g}_{p+h} - \tilde{g}_p +  h \, \tilde{g}_p' \big\|_{L^2(0,L)} \to 0 \quad \text{as } 0 \not= h \to 0,
  $$
  which proves the claim as $p \in \R$ is arbitrary.
  \end{proof}

Since the mapping $\mathbb{R}\ni p \mapsto g_p \in L^2(\partial \Omega)^d$ is Fr\'echet differentiable by virtue of Lemma~\ref{lemma:p_diff} and the mapping $  L^2(\partial \Omega)^d \ni g_p \mapsto u_p \in \mathcal{V}$ is linear and bounded, it immediately follows from the chain rule for Banach spaces that the mapping
$$
N: 
  \left\{
  \begin{array}{l}
     p \mapsto u_p,
    \\[1mm]
  \R \to  \mathcal{V}
  $$
  \end{array}
\right.
$$
is also Fr\'echet differentiable.
\begin{corollary}
  \label{corollary:p_diff}
  The mapping $N:  \R \to  \mathcal{V}$ is Fr\'echet differentiable. The associated derivative at $p \in \R$ is given by the linear map
  $$
  D_p N(p): 
  \left\{
  \begin{array}{l}
  h \mapsto h u_p', \\[1mm]
  \R \to  \mathcal{V},
\end{array}
\right.
$$
where $u_p'$ is the unique solution to \eqref{var_form} with $g_p$ replaced by $-g_p' \in L^\infty(\partial \Omega)^d$.
\end{corollary}

Observe that solving \eqref{var_form} can be interpreted as evaluating the mapping
$$
N: 
  \left\{
  \begin{array}{l}
     (p, \tau) \mapsto u_p, \\[1mm]
   \R \times L_+^\infty(\Omega)^2  \to  \mathcal{V},
  \end{array}
\right.
$$
where we have abused the notation by redefining the operator $N$ to have two arguments. It is well known that $N$ is Fr\'echet differentiable with respect to its second variable as well.
\begin{lemma}
  \label{lemma:tau_deriv}
  The mapping $N: \R \times L^\infty_+(\Omega)^2  \to  \mathcal{V}$ is Fr\'echet differentiable with respect to its second variable. The associated derivative at $(p, \tau) \in \R \times L^\infty_+(\Omega)^2$ is given by the linear and bounded map
  $$
  D_{\tau}N(p, \tau): \left\{
  \begin{array}{l}
  \eta \mapsto D_\tau u_p(\eta), \\[1mm]
   L^\infty(\Omega)^2 \to  \mathcal{V},
\end{array}
  \right.
$$
  where $D_\tau u_p(\eta) \in \mathcal{V}$ is the unique solution of
  \begin{equation}
    \label{eq:tau_deriv}
    B_\tau\big(D_\tau u_p(\eta), v\big) = - B_{\eta}(u_p, v)
  \end{equation}
  for all $v \in \mathcal{V}$.
\end{lemma}

\begin{proof}
  The result follows by utilizing the coercivity of the bilinear form \eqref{eq:bilinear} as well as its linear dependence on $\tau$ in a standard manner;~cf.,~e.g.~\cite{garde2021series}.
  \end{proof}

To complete this section, let us consider the second derivative $D_p D_\tau N(p,\tau)$, which is the tool needed for building rudimentary differentiation-based algorithms for (Bayesian) optimal experimental design in the framework of linear elasticity.

\begin{theorem}
  \label{theorem:second_deriv}
  The mapping $D_\tau N(\, \cdot \,, \tau): \R \to \mathcal{L}(L^\infty(\Omega)^2, \mathcal{V})$ is Fr\'echet differentiable. The associated derivative at $(p, \tau) \in \R \times L^\infty_+(\Omega)^2$ is given by $D_p D_\tau N(p, \tau) \in \mathcal{L}(\R, \mathcal{L}(L^\infty(\Omega)^2, \mathcal{V}))$ defined via 
  $$
  D_p D_\tau N(p, \tau): h \mapsto \big( \eta \mapsto h D_\tau u_p'(\eta) \big),
$$
  where $D_\tau u_p'(\eta) \in \mathcal{V}$ is the unique solution of
  \begin{equation}
    \label{eq:second_deriv}
    B_\tau\big(D_\tau u_p'(\eta), v \big) = -B_{\eta}(u'_p, v) \quad \text{for all } v \in \mathcal{V},
  \end{equation}
  and $u'_p \in \mathcal{V}$ is the unique solution of \eqref{var_form} with $g_p$ replaced by $-g'_p \in L^\infty(\partial \Omega)^d$.
\end{theorem}

\begin{proof}
  Let $\tau \in L^\infty_+(\Omega)^2$ be arbitrary, consider \eqref{eq:tau_deriv} for the location parameters $p$ and $p + h$, and subtract the former from the latter:
  $$
   B_\tau\big(D_\tau u_{p+h}(\eta) - D_\tau u_{p}(\eta), v\big) = B_\eta(u_p-u_{p+h},v) \quad \text{for all } v \in \mathcal{V}.
  $$
   Subtracting \eqref{eq:second_deriv} multiplied by $h$ gives
   $$
   B_\tau\big(D_\tau u_{p+h}(\eta) - D_\tau u_{p}(\eta) - h D_\tau u_p'(\eta), v\big) = B_\eta(u_{p} - u_{p+h} + h u'_p,v). 
   $$
   If one chooses $v = D_\tau u_{p+h}(\eta) - D_\tau u_{p}(\eta) - h D_\tau u_p'(\eta)$ and employs the coercivity \eqref{eq:coer} and continuity \eqref{eq:cont} of the considered bilinear forms, it straightforwardly follows that
   $$
   \big\| D_\tau u_{p+h}(\eta) - D_\tau u_{p}(\eta) - h D_\tau u_p'(\eta) \big\|_{H^1(\Omega)^d} \leq K \|\eta\|_{L^\infty(\Omega)^2} \| u_{p+h} -u_{p} - h u'_p \|_{H^1(\Omega)^d}
   $$
   for some $K = K(\Omega, \Gamma_{\rm D}, \tau)>0$. Assuming $h \not=0$ and $\eta \not= 0$, dividing by $h \|\eta\|_{L^\infty(\Omega)^2}$ and taking the supremum over all $\eta$ satisfying $\| \eta \|_{L^\infty(\Omega)^2} = 1$, we finally get
   \begin{align*}
     \frac{1}{h} \big\| D_\tau N(p+h, \tau) -  D_\tau N(p, \tau) - D_p & D_\tau  N(p, \tau)h \big \|_{\mathcal{L}(L^\infty(\Omega)^2, H^1(\Omega)^d)} \\
     & \leq  \frac{K}{h} \| u_{p+h} -u_{p} - h u'_p \|_{H^1(\Omega)^d},
   \end{align*}
   which tends to zero as $h \to 0$ by virtue of Corollary~\ref{corollary:p_diff}. This concludes the proof.
     \end{proof}

Observe that we could have as well proved that the mapping $D_pN(p, \, \cdot \,): L^\infty_+(\Omega)^2 \to \mathcal{L}(\R, \mathcal{V})$ is Fr\'echet differentiable, with its derivative $D_\tau D_p N(p, \tau) \in \mathcal{L}(L^\infty(\Omega)^2, \mathcal{L}(\R, \mathcal{V}))$ given by
$$
D_\tau D_p N(p, \tau): \eta \mapsto \big( h \mapsto h D_\tau u_p'(\eta) \big).
$$
This result follows from the same argument that leads to Lemma~\ref{lemma:tau_deriv}, but with the boundary condition $g_p$ in \eqref{var_form} replaced by $-g'_p$. In particular, the order in which $N$ is differentiated can be changed, which would also follow by proving that $D_p D_\tau N : \R \times L^\infty_+(\Omega)^2 \to \mathcal{L}(\R, \mathcal{L}(L^\infty(\Omega)^2, \mathcal{V}))$, or $D_\tau D_p N : \R \times L^\infty_+(\Omega)^2 \to \mathcal{L}(L^\infty(\Omega)^2, \mathcal{L}(\R, \mathcal{V}))$, is continuous. In other words, $D_p D_\tau N(p,\tau)$ and  $D_\tau D_p N(p, \tau)$ coincide as bilinear mappings on $\R \times L^\infty(\Omega)^2$.

\begin{remark}
  \label{remark:sampling}
  In the following sections, we are mainly interested in the derivatives of the map
$$
  \gamma N:
  \left\{
  \begin{array}{l}
    (p,\tau) \mapsto {u_p}|_{\Gamma_{\rm N}}, \\[2mm]
    \R \times L^\infty_+(\Omega)^2 \to L^2(\Gamma_{\rm N})^d,
  \end{array}
  \right.
  $$
  where $\gamma: H^1(\Omega)^d \to L^2(\Gamma_{\rm N})^d$ is the bounded Dirichlet trace operator on $\Gamma_{\rm N}$. The required derivatives of this map can be obtained from the results of Corollary~\ref{corollary:p_diff}, Lemma~\ref{lemma:tau_deriv} and Theorem~\ref{theorem:second_deriv} by taking traces of the elements of $H^1(\Omega)^d$ defining the derivatives $D_{p}N$, $D_{\tau} N$ and $D_p D_{\tau} N$. In particular, the Dirichlet boundary values of the solutions to \eqref{eq:tau_deriv} and \eqref{eq:second_deriv} can alternatively be assembled by utilizing the formulas
   \begin{equation}
    \label{eq:tau_deriv_sample}
  \int_{\Gamma_{\rm N}} f \cdot D_{\tau} u_p(\eta) \, {\rm d} s = - B_{\eta}(u_p, u_f)
  \end{equation}
  and 
  \begin{equation}
    \label{eq:second_deriv_sample}
  \int_{\Gamma_{\rm N}} f \cdot D_{\tau} u'_p(\eta) \, {\rm d} s = - B_{\eta}(u'_p, u_f),
  \end{equation}
  where $u_f \in \mathcal{V}$ is the solution of \eqref{var_form} with $g_p$ replaced by $f \in L^2(\Gamma_{\rm N})^d$. These follow straightforwardly by comparing \eqref{var_form} to \eqref{eq:tau_deriv} and  \eqref{eq:second_deriv}, respectively, and they enable solving for all derivatives we need in (Bayesian) optimal experimental design without having to solve any other variational problems than \eqref{var_form}.
\end{remark}


\section{Finite-dimensional linearized forward model}
\label{sec:discretized}

Let us adopt as our measurement model the linearization of the forward map $\gamma N(p, \, \cdot \,): L^\infty_+(\Omega)^2 \to L^2(\Gamma_{\rm N})^d$ around the expected Lam\'e parameter pair $\tau_0 = (\lambda_0, \mu_0) \in L^\infty_+(\Omega)^2$, that is, our aim is to perform optimal Bayesian experimental design with respect to $p \in \R$ assuming that the linearized forward map at $\tau_0$ is an accurate enough measurement model for our purposes. The measurements on the displacement field on $\Gamma_{\rm N}$ are modeled by ``sensor functions'' $s_1, \dots, s_M \in [L^2(\Gamma_{\rm N})^d]^* \cong L^2(\Gamma_{\rm N})^d$, which leads to investigating the behavior of the bounded linear map
\begin{equation}
  \label{eq:inf_forward_map}
  \mathcal{F}(p):
  \left\{
  \begin{array}{l}
    \eta \mapsto \Big[ \big \langle \gamma D_\tau N( p, \tau_0) \eta, s_m \big \rangle_{L^2(\Gamma_{\rm N})^d} \Big]_{m=1}^M, \\[3mm]
    L^\infty(\Omega)^2 \to \R^M,
  \end{array}
  \right.
  \end{equation}
as a function of $p \in \R$. Note that a sensor function may,~e.g.,~measure a weighted mean value of $\nu \cdot D_\tau N( p, \tau_0) \eta$ over some small section of $\Gamma_N$,~i.e.,~a weighted mean of the normal component of the `linearized' relative boundary displacement caused by a pressure activation at $\gamma(p) \in \partial \Omega$ and corresponding to a perturbation $\eta$ in the background Lam\'e parameters $\tau_0$.

The interpretation of measurements as dual evaluations with sensor functions allows one to utilize \eqref{eq:tau_deriv_sample} and \eqref{eq:second_deriv_sample} of Remark~\ref{remark:sampling} to assemble the mapping $\mathcal{F}: \R \to \mathcal{L}( L^\infty(\Omega)^2, \R^M)$ and its derivative. Indeed, by choosing $f = s_m$ in \eqref{eq:tau_deriv_sample} and \eqref{eq:second_deriv_sample}, and referring to Theorem~\ref{theorem:second_deriv}, it follows that
\begin{equation}
  \label{eq:meas_sampling}
\mathcal{F}_m(p) \eta = \langle \gamma D_\tau N( p, \tau_0) \eta, s_m \rangle_{L^2(\Gamma_{\rm N})^d} =  \int_{\Gamma_{\rm N}} s_m \cdot D_{\tau} u_p(\eta) \, {\rm d} s = - B_{\eta}(u_p, u_{s_m})
\end{equation}
and
\begin{equation}
  \label{eq:deriv_sampling}
D_p \mathcal{F}_m(p) \eta = D_p \langle \gamma D_\tau N( p, \tau_0) \eta, s_m \rangle_{L^2(\Gamma_{\rm N})^d} =  \int_{\Gamma_{\rm N}} s_m \cdot D_{\tau} u'_p(\eta) \, {\rm d} s = - B_{\eta}(u'_p, u_{s_m}),
\end{equation}
where $u_{s_m} \in \mathcal{V}$ is the solution of \eqref{var_form} for $g = s_m$ that is assumed to be real-valued. Thus, computing $\mathcal{F}$ and its derivatives at a given $p \in \R$ essentially only requires solving for $u_p$ and $u'_p$, the latter by replacing $g_p$ with  $-g_p'$ in \eqref{var_form}, and then evaluating the bilinear forms on the right-hand sides of \eqref{eq:meas_sampling} and \eqref{eq:deriv_sampling} for all $m=1, \dots, M$ and for the considered perturbations $\eta$.  Observe that the auxiliary solutions $u_{s_m}$, $m=1, \dots, M$, can be precomputed and stored prior to running the optimization algorithm. Note also that all solutions of variational problems appearing in \eqref{eq:meas_sampling} and \eqref{eq:deriv_sampling} correspond to the background Lam\'e parameters $\tau = \tau_0$.

To further simplify the studied model, we assume that both components of the perturbation $\eta$ are given as linear combinations of some basis functions $\psi_1, \dots, \psi_N \in L^\infty(\Omega)$, which enables identifying $\eta$ with an element of $\R^{2N}$. This parametrization can be related to  nodal values in a {\em finite element} (FE) discretization employed when numerically solving for displacement fields in Section~\ref{sec:numerics}, but it can also originate from prior information on the expected behavior of the perturbations in the Lam\'e parameters. As our aim is to optimize a sequence of activation locations $p_1, \dots, p_K$, we abuse the notation by setting $p = (p_1, \dots, p_K) \in \R^K$ and define the finite-dimensional multi-activation version of \eqref{eq:inf_forward_map} via
  \begin{equation}
  \label{eq:forward_map}
  F(p):
  \left\{
  \begin{array}{l}
    \alpha \mapsto \begin{bmatrix} \mathcal{F}(p_1)\eta_\alpha \\ \vdots \\ \mathcal{F} (p_K)\eta_\alpha \end{bmatrix}, \\[10mm]
    \R^{2N} \to \R^{KM},
  \end{array}
  \right.
  \end{equation}
  where
  \begin{equation}
    \label{eq:parametrization}
  \eta_\alpha = \sum_{n=1}^N \big( \alpha_n \psi_n, \alpha_{N+n} \psi_n \big) \in L^\infty(\Omega)^2.
  \end{equation}
  Take note that $F: \R^{K} \to \R^{KM \times 2N}$ is a (continuously) differentiable map that can be evaluated, together with its partial derivatives, based on \eqref{eq:meas_sampling}, \eqref{eq:deriv_sampling} and \eqref{eq:forward_map}.

  \begin{remark}
    One could as well optimize the locations (or other properties) of the sensor functions $s_1, \dots, s_M$. In fact, even the required analysis would be essentially the same as that for the positions of the activations presented above due to the self-adjointness of the considered elliptic partial differential equation. Be that as it may, only the optimization of the activation locations is considered in the numerical experiments of Section~\ref{sec:numerics}.
    \end{remark}

\section{Bayesian inversion and experimental design}
\label{sec:bayes}

Assume the linearized finite-dimensional measurement model \eqref{eq:forward_map} and an additive noise process, which means that a noisy measurement $y \in \R^{KM}$ can be written as
\begin{equation}
  \label{eq:meas_model}
y = F(p) \alpha + \omega,
\end{equation}
where $p \in \R^K$ is a vector of activation positions, $\alpha \in \R^{2N}$ defines the Lam\'e parameters via \eqref{eq:parametrization}, and $\omega \in \R^{KM}$ models the measurement noise. In Bayesian inversion, $\alpha$, $y$ and $\omega$ are treated as random variables. Our prior information on the unknown of primary interest $\alpha$ is encoded in the prior probability density $\pi_{\rm pr}: \R^{2N} \to \overline{\R_+}$. According to the Bayes' formula, the posterior density for $\alpha$ reads
\begin{equation}
  \label{eq:Bayes}
\pi(\alpha \, | \, y;  p) = \frac{\pi(y \, | \, \alpha;  p) \pi_{\rm pr}(\alpha)}{\pi(y; p)},
\end{equation}
where $\pi(y \, | \, \cdot \, ;  p): \R^{2N} \to \overline{\R_+}$ is the likelihood function. Here and in what follows, each $\pi$ denotes a probability density, the precise interpretation of which should be clear from the context.

We assume that the prior and noise are independent Gaussians,~i.e.,~$\alpha \sim \mathcal{N}(0, \Gamma_{\rm pr})$ and $\omega \sim \mathcal{N}(0, \Gamma_{\rm noise})$, where $\Gamma_{\rm pr} \in \R^{2N \times 2N}$ and $\Gamma_{\rm noise} \in \R^{KM \times KM}$ are symmetric and positive definite covariance matrices. Assuming a zero mean for $\alpha$ can be motivated by appropriately choosing the background Lam\'e parameter pair $\tau_0$. On the other hand, if the mean of the noise were not initially zero, it could be subtracted from both sides of \eqref{eq:meas_model}, thus redefining the measurement and a new zero-mean noise term. Be that as it may, neither of these means affects the target function of A-optimality introduced below.

Under the above assumptions, the posterior density \eqref{eq:Bayes} is also Gaussian, with the mean and covariance
\begin{subequations}
   \label{eq:posterior_cov2}
   \begin{align}
     \widehat{\alpha}(p) &=  \Gamma_{\rm pr} F(p)^{\top} \big(F(p) \Gamma_{\rm pr} F(p)^{\top} + \Gamma_{\rm noise}  \big)^{-1} y, \\[2mm]
     \label{eq:posterior_cov}
  \Gamma_{\rm post}(p) &= \Gamma_{\rm pr} - \Gamma_{\rm pr} F(p)^{\top} \big(F(p) \Gamma_{\rm pr} F(p)^{\top} + \Gamma_{\rm noise}  \big)^{-1} F(p) \Gamma_{\rm pr} ,
   \end{align}
   \end{subequations}
respectively \cite{Kaipio06}. By using the Woodbury matrix identity, \eqref{eq:posterior_cov2} could be transformed into a form that involves the inversion of $\Gamma_{\rm post}(p)^{-1}= \Gamma_{\rm prior}^{-1} + F(p)^{\top} \Gamma_{\rm noise}^{-1} F(p) \in \R^{2 N \times 2 N}$ instead of $F(p) \Gamma_{\rm pr} F(p)^{\top} + \Gamma_{\rm noise} \in \R^{KM \times KM}$. However, in our numerical experiments the dimensions $2N$ are $KM$ are so low that the choice between these formulations is not essential.

\subsection{A-optimality}
\label{sec:A_and_D}
An A-optimal experimental design minimizes the expected squared distance from the mean of the posterior to the unknown of primary interest in a given seminorm. Assume that the considered seminorm corresponds to the positive semidefinite weight matrix $A^{\top} \!A$, with $A \in \R^{2 N \times 2 N}$. It can be straightforwardly deduced that in our setup, forming an A-optimal design is equivalent to finding $p_{\rm A} \in \R^K$ that satisfies (see,~e.g.,~\cite[Appendix~A]{Burger21})
\begin{equation}
\label{eq:Aoptimal}
p_{\rm A} = {\rm arg} \min_{p \in \R^K} \Phi_{\rm A}(p),
\end{equation}
with
\begin{equation}
  \label{eq:Atarget}
\Phi_{\rm A}(p)  := {\rm tr}  \big(A \Gamma_{\rm post}(p) A^{\top}\big) = {\rm tr}  \big(\Gamma_{\rm post}(p) A^{\top} \! A\big),
\end{equation}
where the latter equality follows from the invariance of the matrix trace under cyclic perturbations. The matrix $A$ is used to weigh the expected reconstruction error,~i.e.,~the expected difference between the posterior mean and the underlying true unknown, differently in different directions. As examples, choosing $A$ to be the identity matrix $I$ corresponds to using the squared Euclidean norm as the measure of reconstruction error, and setting $A = I_{\rm ROI}$, with the diagonal matrix $I_{\rm ROI}$ having ones as the diagonal elements for the degrees of freedom corresponding to a {\em region of interest} (ROI) and zeros as the other elements, leads to only considering the squared Euclidean error over the ROI. Moreover, if the basis functions in \eqref{eq:parametrization} correspond to a FE discretization, it may be reasonable to choose $A^{\top} \!A$ as the corresponding mass matrix, so that the reconstruction error is approximately measured in the squared norm of $L^2(\Omega)$. 

The A-optimization target $\Phi_{\rm A}(p)$ depends on the activation positions through \eqref{eq:posterior_cov} and \eqref{eq:forward_map}. Hence, its gradient can be straightforwardly, but tediously, calculated by applying standard differentiation formulas of matrix functions to \eqref{eq:Atarget} and \eqref{eq:posterior_cov} and then utilizing \eqref{eq:meas_sampling}. We do not present these calculations and formulas here but instead refer to,~e.g.,~the master's thesis \cite{Pohjavirta21} for further details.

\section{Implementation}
\label{sec:implementation}
This section introduces the setting for our numerical tests. We begin with definitions needed for describing the deterministic setup,~i.e.,~the domain $\Omega$, the pressure activations, the measurement sensors and the parametrization for the unknown. Subsequently, we briefly consider the prior and noise distributions as well as the weight matrix $A$ for the A-optimality criterion. Finally, the implementation of a gradient descent algorithm employed in some of our examples is discussed.

\begin{figure}[t]
\centering 
\includegraphics[width=0.49\textwidth]{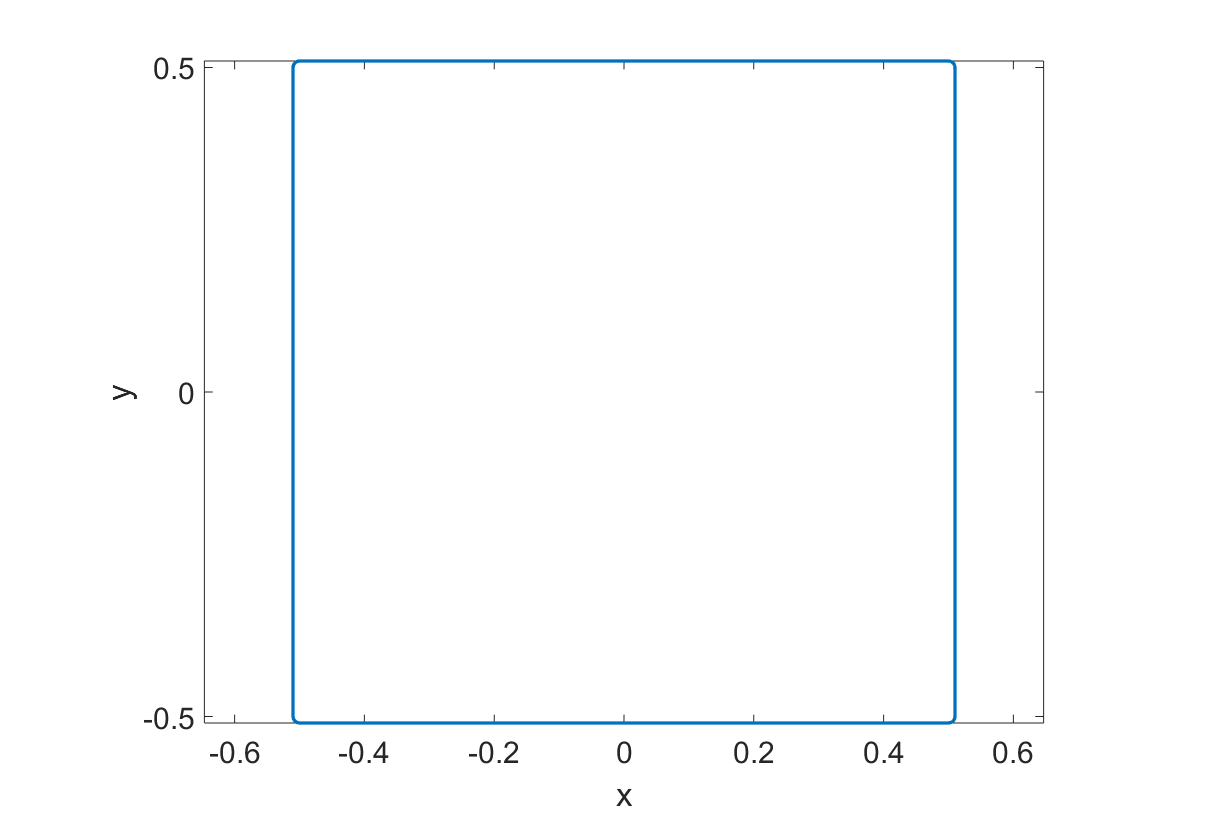}
\includegraphics[width=0.49\textwidth]{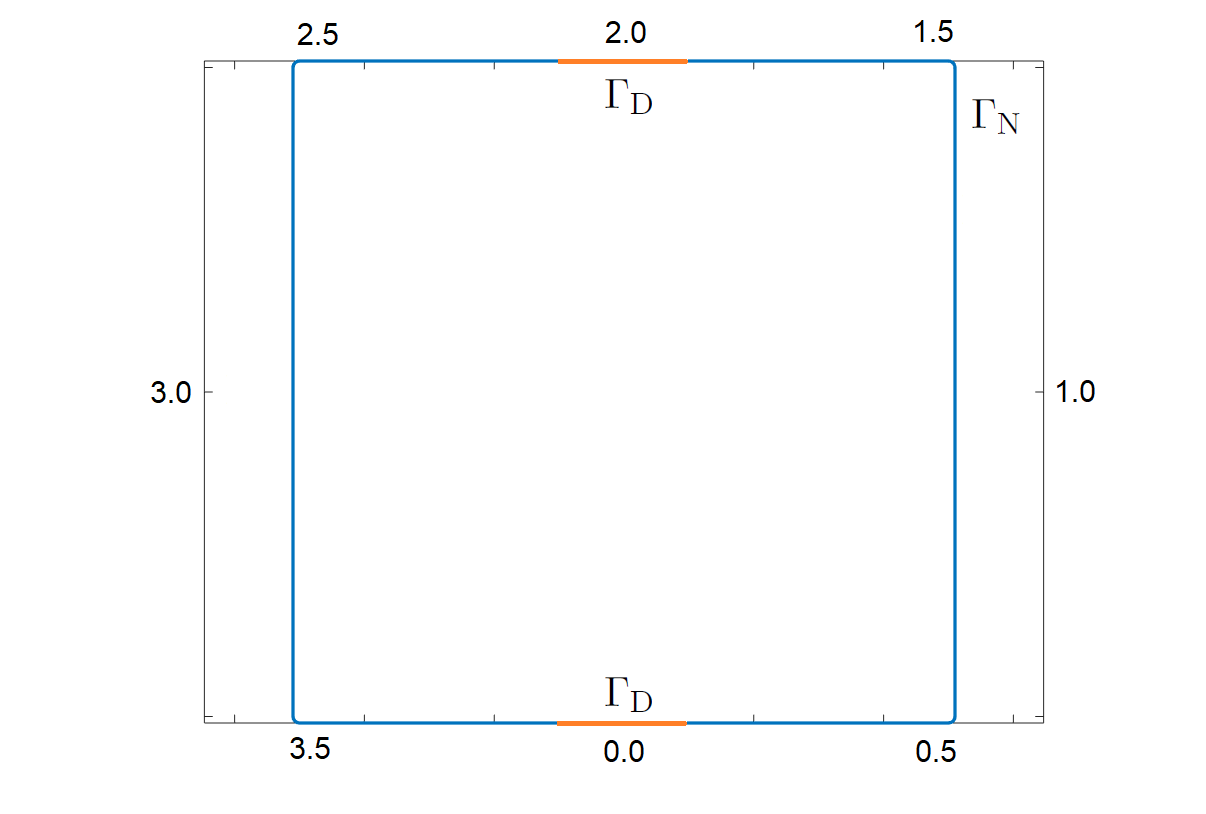}
\caption{Test object $\Omega$,~i.e.,~a unit square with rounded corners. Left: $\Omega$ embedded in $\R^2$. Right: Arclength parametrization of $\Omega$ including the Neumann (blue) and Dirichlet (orange) boundaries.}\label{setting_test_object}\label{param_testobject}
\end{figure}
\noindent

\subsection{Deterministic setup}
The computational domain $\Omega$ is the slightly rounded square shown in Figure~\ref{setting_test_object}, with the unit of length being meter. The circumference of $\Omega$ is $L= 4 + 2\pi r$, where $r=10^{-3}$ is the radius of the circular arcs that smoothen the corners of the square. The arclength parametrization $\gamma: [0, L) \to \R^2$ of the boundary $\partial \Omega$ in the counter clockwise direction is explicitly written down in Appendix~\ref{app:A}; the most essential detail to note here is that the parameter value $t=0$ corresponds to the midpoint of the bottom edge of $\Omega$. The left-hand image of Figure~\ref{setting_test_object} presents $\Omega$ embedded in $\R^2$, whereas the right-hand image shows the arclength parameters for a few boundary points and visualizes the Dirichlet and Neumann boundaries. The former is composed of boundary segments of length $0.2$ at the middle of the top and bottom edges of $\Omega$.

Let $\tilde{\nu} = \nu \circ \gamma$ denote the arclength parametrization for the exterior unit normal of~$\Omega$. The shape of our arclength-parametrized activation field $\tilde{g} = g \circ \gamma: \R \to \R^2$ is
\begin{equation}
  \label{eq:no_move_g}
\tilde{g}(t) = \exp\left(-\frac{\left(\cos\left(\frac{2\pi}{L}t \right)-1\right)^2}{2\sigma^2}\right) \tilde{\nu}(t), \qquad \sigma > 0,
\end{equation}
which is a smooth $L$-periodic vector field normal to $\partial \Omega$. The absolute value of $\tilde{g}$ reaches its maximum value of $1$ only at $t=0$ (modulo $L$), and it is symmetric with respect to $t=0$, bearing some resemblance to a Gaussian bell curve with a standard deviation $\sigma > 0$. The $p$-transferred version $\tilde{g}_p: \R \to \R^2$ is defined as
\begin{equation}
  \label{eq:move_g}
\tilde{g}_p(t) = \exp\left(-\frac{\left(\cos\left(\frac{2\pi}{L}(t-p) \right)-1\right)^2}{2\sigma^2}\right) \tilde{\nu}(t) \, ,
\end{equation}
the absolute value of which attains its maximum at $p \in \R$. In particular, $p \in [0, L)$ gives the distance along $\partial \Omega$ in the counter clockwise direction from the midpoint of the bottom edge of $\Omega$ to the point of highest pressure in the boundary activation $g_p = \tilde{g}_p \circ \gamma^{-1}|_{\partial \Omega}$.  In precise mathematical terms, the support of $g_p$ is $\partial \Omega$, but from the numerical standpoint, it vanishes on most of $\partial \Omega$ if the standard deviation $\sigma$ is small. The pressure field $g_p$ for $p = 1+\tfrac{\pi}{2}r$ is depicted in Figure~\ref{g_different_supports} for $\sigma = 0.001, 0.01$ and $0.05$, of which $\sigma = 0.01$ is the value used in our numerical studies. 

  \begin{remark}
    The $p$-transferred boundary pressure $g_p$ defined by \eqref{eq:move_g} is not strictly speaking compatible with the construction in Section~\ref{sec:frechet} since the unit normal in \eqref{eq:move_g} is not translated by $-p$. However, it is easy to check that the analysis of Section~\ref{sec:frechet} remains valid for such $g_p$ if the $p$-derivative of the pressure field $g_p'$ is (re)defined by only differentiating and translating the scalar multiplier of the unit normal in \eqref{eq:no_move_g} and leaving the unit normal itself untouched.
    \end{remark}
  
\begin{figure}[t]
\centering 
\includegraphics[width=1\textwidth]{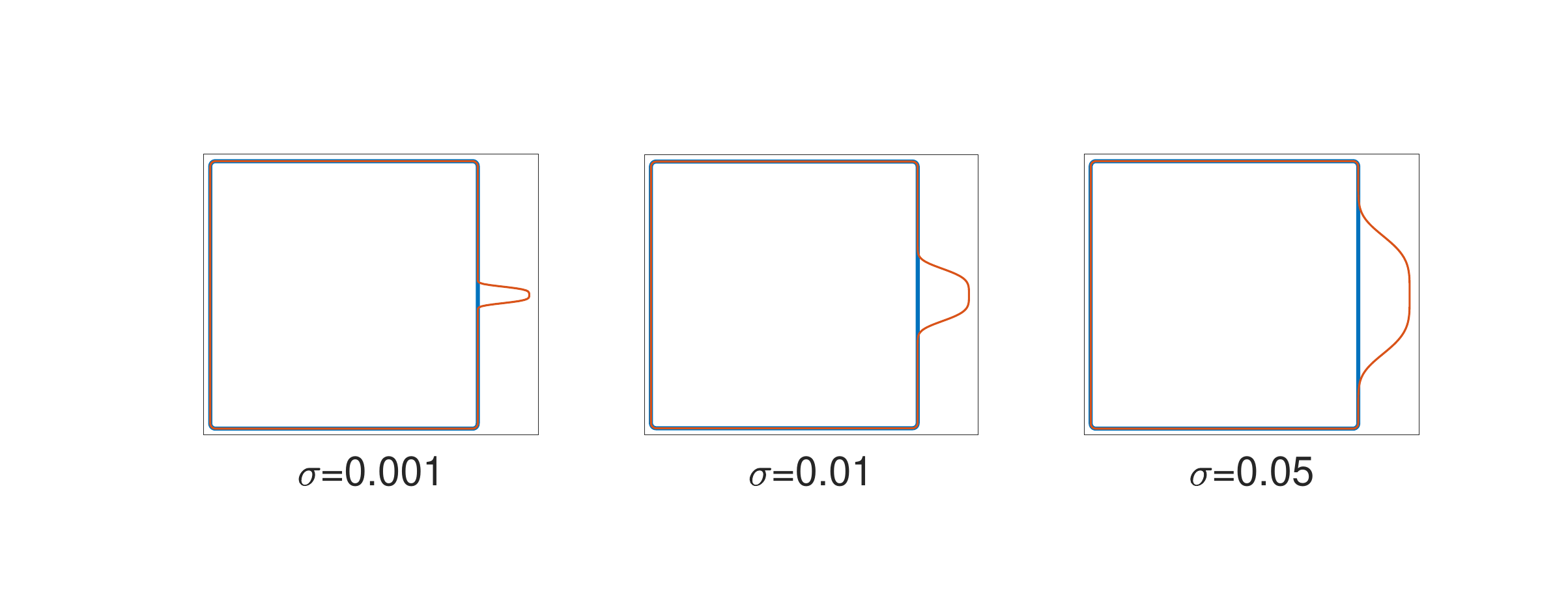}
\caption{The pressure field $g_p$ for $p=1+\tfrac{\pi}{2}r$ and three values for $\sigma$.}\label{g_different_supports}
\end{figure}
\noindent

The sensor functions, employed in Section~\ref{sec:discretized} to define the discrete measurements, are chosen to be
\begin{align*}
s_m=g_{p_m},\quad m=1,\ldots,M,
\end{align*}
\noindent
where $p_m$ are parameters corresponding to equidistant points on $\Gamma$, with the number of measurements fixed to $M=20$ in the examples of Section~\ref{sec:numerics}. The considered measurements are thus weighted averages of the normal displacement field on $\partial \Omega$.

What remains to be chosen is the parametrization for $\eta$,~i.e.,~for the perturbations in $\lambda$ and $\mu$ around their background values. Let us assume that $\lambda$ and $\mu$ are naturally represented in some FE basis as
\begin{align*}
  \lambda =  \sum_{n=1}^{\widetilde{N}} \lambda_n \phi_n,
   \qquad 
\mu = \sum_{n=1}^{\widetilde{N}} \mu_n \phi_n, \qquad \lambda_n, \mu_n\in\mathbb{R}.
\end{align*}
\noindent
A computationally straightforward option is to choose the functions $\psi_1,\ldots, \psi_N$ defining the Lam\'e perturbation in \eqref{eq:parametrization} as sums of the finite element basis functions $\phi_1,\ldots,\phi_{\widetilde{N}}$:
\begin{align}\label{psi}
  \psi_n = \sum_{\widetilde{n} \in \mathcal{N}_n} \phi_{\widetilde{n}}, \qquad n = 1, \dots, N \leq \tilde{N},
\end{align}
where $\mathcal{N}_1, \dots, \mathcal{N}_N$ are disjoint subsets of $\{ 1, \dots, \widetilde{N} \}$. In our numerical studies, $N = 50$ and the subsets $\mathcal{N}_1, \dots, \mathcal{N}_N$ are chosen such that $\psi_1, \dots, \psi_N$ approximate the characteristic functions of the triangular subdomains of $\Omega$ shown in Figure~\ref{subdomains}. Take note that these subdomains have approximately the same area.

\begin{figure}[t]
\centering 
\includegraphics[width=0.4\textwidth]{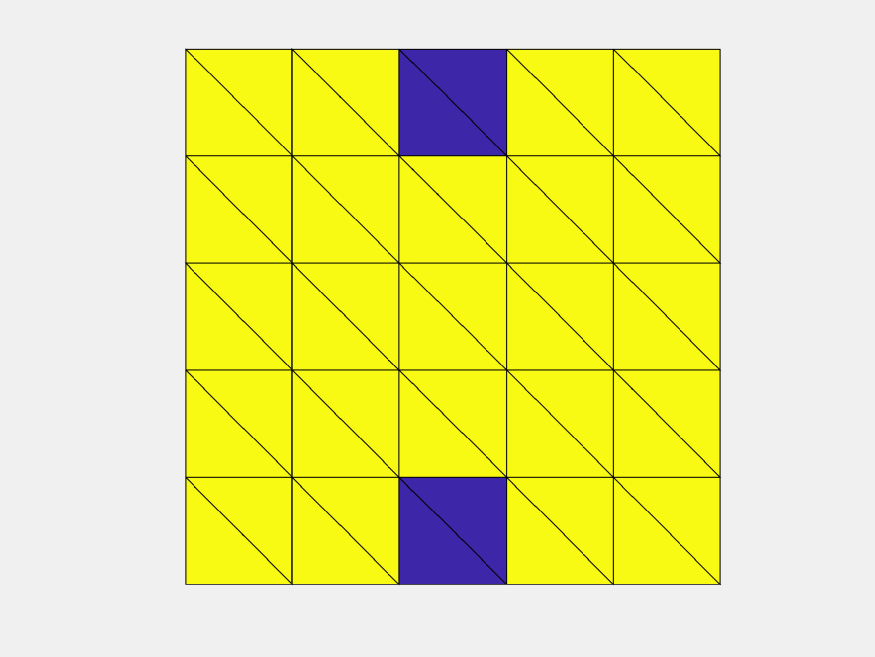}
\caption{Subdomains of $\Omega$ employed in an (approximately) piecewise constant parametrization for the Lam\'e parameter perturbation $\eta$.}\label{subdomains}
\end{figure}

Assume the above choices and definitions. For a single activation position $p \in \R$, the elements in the matrix $F(p)\in \mathbb{R}^{M\times 2N}$ (cf.~\eqref{eq:forward_map}) and those in its elementwise derivative $F'(p) = \tfrac{\partial}{\partial p} F(p)\in \mathbb{R}^{M\times 2N}$ can be numerically evaluated via the following steps, of which the zeroth need not be repeated for new values of $p$:

\bigskip

\begin{itemize}
\item[0.] As initialization, solve for $u_{s_m}$, $m=1,\ldots,M$, at $\tau = \tau_0$.
  \smallskip
\item[1.] Compute $F(p)$:
\begin{itemize}
\item[(i)] Solve for $u_p$ at $\tau = \tau_0$, if it does not equal an already computed $u_{s_m}$.
\item[(ii)] Employing \eqref{eq:forward_map}, compute
\begin{align*}\label{F}
F_{m, n}(p)=-B_{(\psi_n,0)}(u_p,u_{s_m}) \quad \text{and} \quad F_{m, n+N}(p)=-B_{(0,\psi_n)}(u_p,u_{s_m})
\end{align*}
\noindent
for $m=1,\ldots,M$ and $n=1,\ldots,N$.
\end{itemize}
\smallskip
\item[2.] Compute the derivative $F'(p)$:
\begin{itemize}
\item[(i)] Solve for $u_p^\prime$ at $\tau = \tau_0$.
\item[(ii)] Employing (\ref{eq:deriv_sampling}), compute 
\begin{align*}
 F_{m,n}'(p) = -B_{(\psi_n,0)}(u^\prime_p, u_{s_m}) \quad \text{and} \quad F_{m,N+n}'(p) =-B_{(0,\psi_n)}(u^\prime_p,u_{s_m})
\end{align*}
\noindent
for $m=1,\ldots,M$ and $n=1,\ldots,N$.
\end{itemize}
\end{itemize}

\bigskip

When considering $K$ activations, the complete system matrix $F(p_1, \dots, p_K) \in \mathbb{R}^{KM\times 2N}$ can be built by stacking $F(p_1), \dots, F(p_K)$; see Section~\ref{sec:discretized} and especially~\eqref{eq:forward_map}. Note that only one of the submatrices forming $F(p_1, \dots, p_K)$ depends on the location parameter $p_l$ of any single boundary pressure activation.

The above scheme was implemented in Comsol using approximately $1632$ triangles and $6$ piecewise quadratic basis functions for each triangle for solving the required variational problems. In particular, the Dirichlet boundary was implemented in Comsol, and it was given no further attention during the optimization process, apart from the blue subdomains in Figure \ref{subdomains} being excluded from the ROI, as explained in more detail in the following section. (If the support of a pressure activation $g_p$ overlaps with $\Gamma_{\rm D}$, the part intersecting the Dirichlet boundary is ignored in the computations in accordance with the model introduced in Section~\ref{sec:real_forward}.)

\subsection{Framework for Bayesian OED}
\label{sec:priors}
With the linearized forward operator $F(p)$ in hand, the entities that still need to be defined to evaluate the A-optimality target function in \eqref{eq:Atarget} are the noise $\Gamma_{\rm noise}$ and prior $\Gamma_{\rm pr}$ covariances as well as the weight matrix $A$. The former two are required for forming the posterior covariance via \eqref{eq:posterior_cov}.

We assume a block diagonal prior covariance matrix ${\rm diag}(\gamma_\lambda^2 \Gamma_0, \gamma_\mu^2 \Gamma_0) \in \R^{2N \times 2N}$, where $\Gamma_0$ is defined elementwise as
\begin{equation}
 \label{eq:prior_cov}
 (\Gamma_{\rm prior})_{i,j} = \exp \left(-\frac{| x_i - x_j |^2}{2\ell^2} \right), \qquad i, j = 1, \dots, N,
 \end{equation}
with $x_i$ and $x_j$ denoting the midpoints of the subdomains $i$ and $j$ in Figure~\ref{subdomains}. The parameters $\gamma_\lambda, \gamma_\mu>0$ are the subdomain-wise standard deviations for the respective Lam\'e parameters, and $\ell>0$ is the so-called correlation length that controls {\em a priori}  spatial variations in the perturbations of the parameters. This model assumes no correlation between the two parameters, but it expects their spatial variations to be of similar nature as indicated by the common correlation length. In our numerical experiments, we choose $\ell = 0.1$ and although the distinct standard deviations would allow for defining different scales for the perturbations in $\lambda$ and $\gamma$, we set $\gamma_\lambda = \gamma_\mu = 1$ for simplicity. On the other hand, the components of the additive noise process are assumed to be uncorrelated with a common variance,~i.e.,~$\Gamma_{\rm noise} = \sigma^2 I \in \R^{KM\times KM}$, where the value $\sigma^2= 10^{-3}$ is used throughout the numerical experiments.

As mentioned in Section~\ref{sec:A_and_D}, the weight matrix $A$ in \eqref{eq:Atarget} could be,~e.g.,~chosen to define a ROI inside $\Omega$ or account for the different sizes of the subdomains associated with the degrees of freedom in  the parametrization for the Lam\'e parameter perturbations. As the subdomains in Figure~\ref{subdomains} are approximately of the same size, we do not need to worry about the latter aspect. However, to avoid certain instability issues, we exclude the subdomains lying the closest to $\Gamma_{\rm D}$ from the ROI. That is, we define $A = I_{\rm ROI}$ as an identity matrix with its diagonal elements corresponding to the indices of the four blue subdomains in Figure~\ref{subdomains} replaced by zeros.

\subsection{Gradient descent algorithm}
\label{sec:grad_desc}
As mentioned at the end of Section~\ref{sec:A_and_D}, being able to differentiate the forward operator $F(p)$ with respect to the design parameter vector $p$ also enables computing the gradient for the A-optimality target function $\Phi_{\rm A}(p)$ of \eqref{eq:Atarget} by applying matrix differentiation formulas. The gradient can then be used in implementing a gradient descent algorithm for minimizing $\Phi_{\rm A}(p)$ with respect to (some components of) $p$. The basic idea of gradient descent is, of course, to proceed from the current iterate $p^{(k)}$ to the direction of the negative gradient, say, $d^{(k)}$. The only implementation details that need still to be settled are the employed method for line search in the direction $d^{(k)}$ and the stopping criteria for the algorithm --- the employed initial guesses are defined in the numerical studies of Section~\ref{sec:numerics}. We refer to \cite{Nocedal06} for more information on gradient descent and other methods of numerical optimization.

After computing the negative gradient direction $d^{(k)}$ at the current iterate $p^{(k)}$, a line search on an equidistant grid of $50$ points is performed over a line segment of length $5^{-1} L$ in the direction of $d^{(k)}$. The grid point that produces the largest decrease in the value of $\Phi_{\rm A}$ is dubbed $p^{(k+1)}$. If no improvement in the target function value is observed, the length of the line segment is reduced by a factor of $5^{-1}$, and the search is repeated with same number of grid points. If necessary, this procedure of reducing the search interval is repeated $5$ times. If no reduction in the target function is observed, the whole algorithm is terminated.

There are also two other stopping criteria that are monitored. The algorithm is terminated if either
  $$
  \frac{\Phi_{\rm A}(p^{(k)}) - \Phi_{\rm A}(p^{(k+1)})}{\Phi_{\rm A}(p^{(k)})} <  10^{-4}
  $$
  or
  $$
  \big| p^{(k)} - p^{(k+1)} \big|  < 10^{-3}
  $$
  for $5$ times in a row. There are naturally also many other options for the implementation of the line search step in gradient descent and for the stopping criteria. We do not claim that our choices are optimal, but they seem to function well enough in our setting.

\section{Numerical experiments}
\label{sec:numerics}

This section presents our numerical examples in the setting introduced above. First, only three pressure activations are considered, and an exhaustive search is compared to a greedy sequential approach, with an option to enhance the latter with a simple heuristic that guarantees finding at least a local minimum. Next, the number of activations is increased to ten and the (enhanced) sequential search is compared to gradient descent. The section is completed with a simple demonstration on the effect of the background material parameters on optimal designs.

\subsection{Comparison of exhaustive and sequential search}
\label{sec:num1}

In our first numerical experiment, the background Lam\'e parameter values are chosen as $\mu_0=1.1852 \cdot 10^9$\,Pa and $\lambda_0=2.7654\cdot 10^9$\,Pa, which could model acrylplastic. The exhaustive search for optimal positions of three boundary pressure activations is performed on a $J\times J\times J$ equidistant mesh of arclength parameter points on $[0, L)^3$, with $J=200$. Note that the obvious symmetry in the roles of the pressure activations can be exploited so that one does not actually need to evaluate $\Phi_{\rm A}$ at all $8\cdot 10^6$ mesh points, but the exhaustive search can be carried out with only $\tfrac{1}{6}J(J+1)(J+2) = 1.3534\cdot 10^6$ evaluations of the target function. Be that as it may, the number of required function evaluations grows in any case as $O(J^K)$, with $K$ denoting the number of pressure activations in the design, which makes exhaustive search unsuited for OED with a high number of pressure activations.

  \begin{figure}[t]
\centering 
\includegraphics[width=1\textwidth]{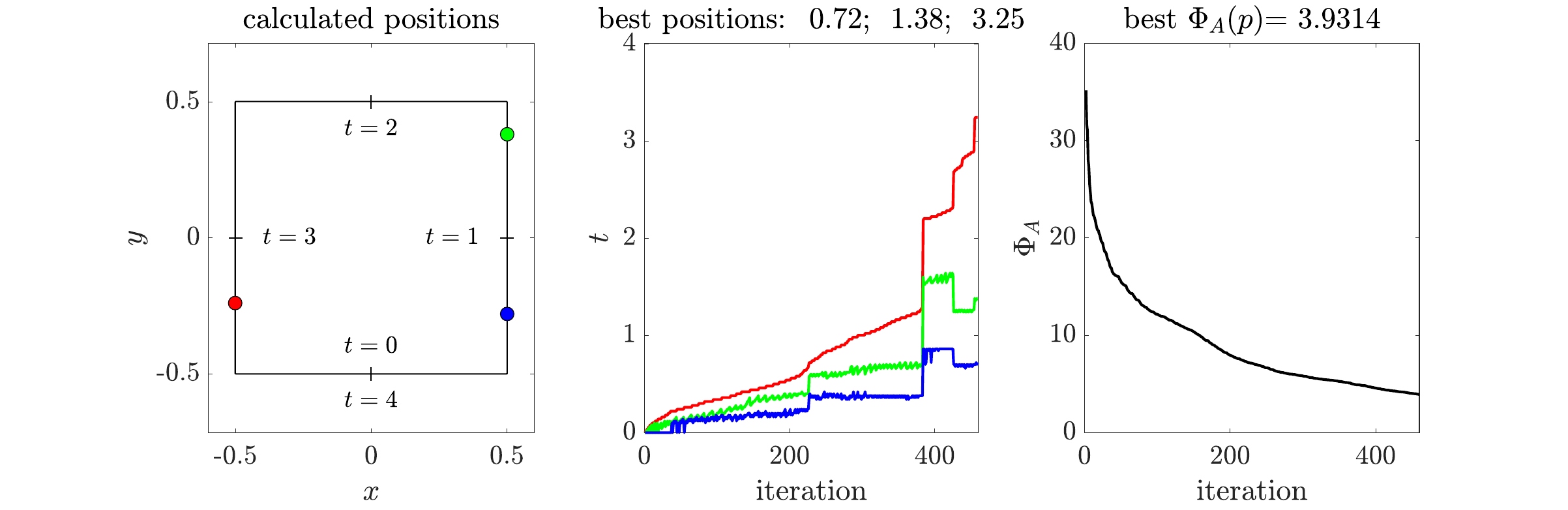}
\caption{Exhaustive search on a $200\times 200\times 200$ grid. Left: Optimal design of three pressure activations corresponding to the arclength parameter triplet $p^* = (0.72, 1.38, 3.25)$. Middle: Progress of the search, with an ``iteration'' referring to the instances when the estimate for $p^*$ was updated in the exhaustive search. Right: Evolution of the optimization target, with the final optimal value $\Phi_{\rm A}(p^*)=3.93$.
}\label{exhaustive}
\end{figure}

  The results of the exhaustive search are visualized in Figure~\ref{exhaustive}. The left-hand image shows the optimal positions for the activations, whereas the middle image illustrates the progress of the algorithm with the same color-coding for the activations as in the left-hand image. Throughout the search, the order of the arclength parameters for the three activations is maintained the same: each time $\Phi_{\rm A}$ is evaluated, the red activation lies the furthest from and the blue activation the closest to $t=0$ in the counter clockwise direction. In the middle image of Figure~\ref{exhaustive}, the vertical axis refers to the arclength parameters of the three activations, and each ``iteration'' on the horizontal axis indicates an evaluation $\Phi_{\rm A}$ that was smaller than its best previously stored value. The right-hand image shows the evolution of the best recorded evaluation for $\Phi_{\rm A}$ as a function of these iterations, leading in the end to the globally optimal value $\Phi_{\rm A}(p^*) = 3.93$. Here and in what follows, we denote by $p^*$ the optimized design parameter vector produced by the considered algorithm that should be clear from the context. Notice that due to symmetry, there must actually be (at least) four equally optimal experimental designs, of which the algorithm found one.

\begin{figure}[t]
\centering 
\includegraphics[width=1\textwidth]{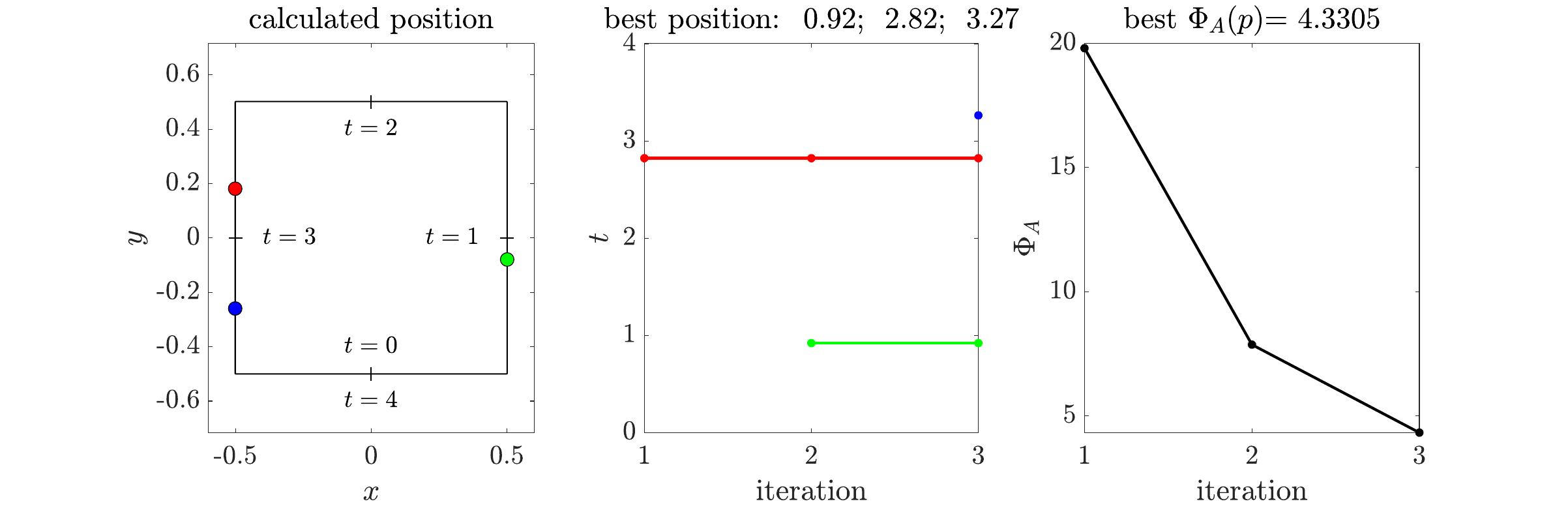}
\caption{Greedy sequential algorithm. Left: Optimized design for three pressure activations corresponding to the arclength parameter triplet $p^* = (0.92, 2.82, 3.27)$. Middle: Progress of the algorithm, with each ``iteration'' referring to an introduction of a new pressure activation whose position has been optimized by a one-dimensional exhaustive search. Right: Evolution of the optimization target, with the final value $\Phi_{\rm A}(p^*)=4.33$.
}\label{sequal_without_line_search}
\end{figure}  

Figure~\ref{sequal_without_line_search} presents the results of a greedy sequential search for finding positions for the three pressure activations, cf.~\cite{Burger21}. The idea is to consider $\Phi_{\rm A}$ in turns as a function of only one of the three activation positions, performing a one-dimensional exhaustive search on the grid of $J$ points corresponding to that variable, fixing the position of the activation to be the found minimizer, and then proceeding to introducing and selecting the position for the next activation. On the positive side, the required number for evaluations of $\Phi_{\rm A}$ in such a sequential algorithm is only $O(JK)$, but on the negative side, there is no guarantee that even a local optimum is found with respect to all activation positions. The left-hand image of Figure~\ref{sequal_without_line_search} shows the found positions for the three activations, the middle image reveals the order in which the activations were introduced in the sequential algorithm (with an ``iteration'' now referring to an incorporation of a new activation to the design), and the right-hand image depicts the evolution of the optimization target. Although the sequential algorithm is able to produce a design that corresponds to a much lower value for $\Phi_{\rm A}$ than a random setup of three pressure activations (cf.~the right-hand image of Figure~~\ref{exhaustive}), the final value of $\Phi_{\rm A}(p^*) = 4.33$ is still $10$\% higher than the actual minimum of $3.93$ found by the exhaustive search.

In order to increase the accuracy of the sequential search without too heavily compromising its computational attractivity, we combine it with gradient descent in a sequential manner: each time after performing a one-dimensional exhaustive search to add a new pressure activation to the design, the positions of {\em all} activations introduced by that point  are fine-tuned by running the gradient descent algorithm described in Section~\ref{sec:grad_desc}. This extra step guarantees that the design of pressure activations corresponds to a local minimum of $\Phi_{\rm A}$ before a new activation is added or the whole algorithm is terminated. The performance of this {\em enhanced sequential search} is analyzed in Figure~\ref{sequal_with_line_search}, which is organized in the same way as Figure~\ref{sequal_without_line_search} for the standard sequential algorithm, apart from ``iteration'' on the horizontal axes of the middle and right-hand images referring to either a one-dimensional exhaustive search or a step in the gradient descent algorithm. The optimized pressure activation design shown in the left-hand image closely resembles the one resulting from the exhaustive search in Figure~\ref{exhaustive}, and the value for the A-optimality target function attained at the end of the algorithm,~i.e.,~$\Phi(p^*)=4.01$, is only about $2$\% higher than that resulting from the exhaustive search.

\begin{figure}[t]
\centering 
\includegraphics[width=1\textwidth]{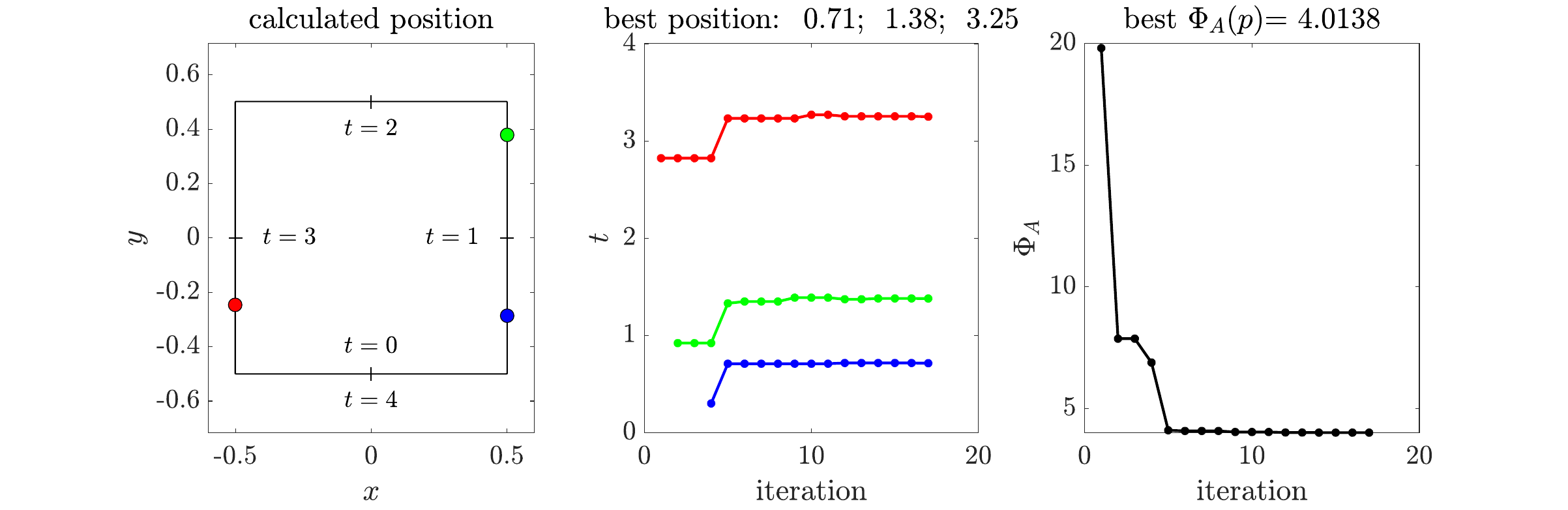}
\caption{Sequential algorithm enhanced by gradient descent. Left: Optimized design for three pressure activations corresponding to the arclength parameter triplet $p^* = (0.71, 1.38, 3.25)$. Middle: Progress of the algorithm, with an ``iteration'' referring to the introduction of a new pressure activation via a one-dimensional exhaustive search or a step of gradient descent for fine-tuning the design after such an introduction. Right: Evolution of the optimization target, with the final value $\Phi_{\rm A}(p^*)=4.01$.
}\label{sequal_with_line_search}
\end{figure}

\subsection{Sequential search and gradient descent for several activations}
In the second experiment, we continue to work with the same background Lam\'e parameter values $\mu_0=1.1852 \cdot 10^9$\,Pa and $\lambda_0=2.7654\cdot 10^9$\,Pa. However, the number of pressure activations is increased to $K=10$ to allow a more practical setting. Three optimization approaches are compared: the basic greedy sequential search, the enhanced sequential search and an application of gradient descent to an initial guess of equidistant activations. The implementations of the sequential algorithms are the same as in the previous section, whereas the gradient descent algorithm proceeds as described in Section~\ref{sec:grad_desc}.

The results for the basic greedy sequential search are presented in Figure~\ref{10_acti_20_sensor_without_line_search}. The figure is organized in the same way as Figure~\ref{sequal_without_line_search}, but it considers a higher number of sequential introductions of new pressure activations. In particular, the middle image indicates the order in which the activations are added to the design. The algorithm seems to generally prefer a relatively uniform grid of activations, but it places less emphasis on the top and bottom edges of $\Omega$ that contain the Dirichlet boundary. The optimized design corresponds to the value $\Phi_{\rm A}(p^*) = 0.87$ that is about a fifth of the value attained with the basic sequential search with only three activations; note that the first three iterations in Figure~\ref{10_acti_20_sensor_without_line_search}, in fact, exactly correspond to the results presented in Figure~\ref{sequal_without_line_search}.

\begin{figure}[t]
\centering 
\includegraphics[width=1\textwidth]{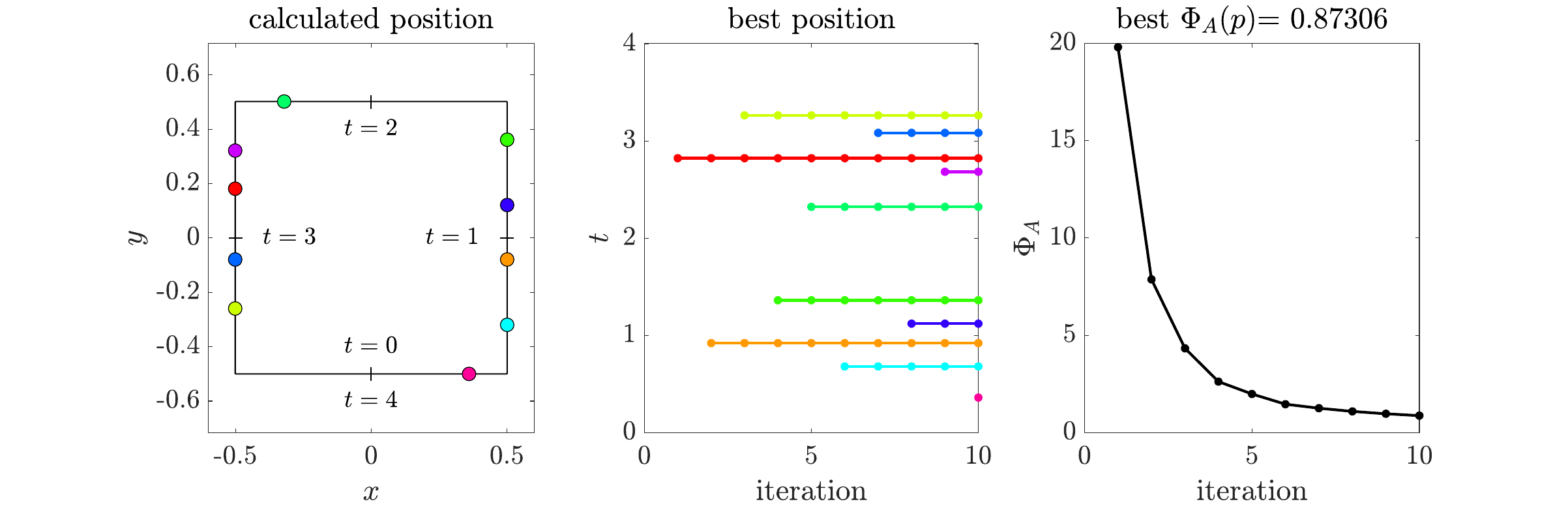}
\caption{Greedy sequential algorithm. Left: Optimized design for ten pressure activations. Middle: Progress of the algorithm, with each ``iteration'' referring to an introduction of a new pressure activation whose position has been optimized by a one-dimensional exhaustive search. Right: Evolution of the optimization target, with the final value $\Phi_{\rm A}(p^*)=0.87$.
}\label{10_acti_20_sensor_without_line_search}
\end{figure}

\begin{figure}[t]
\centering 
\includegraphics[width=1\textwidth]{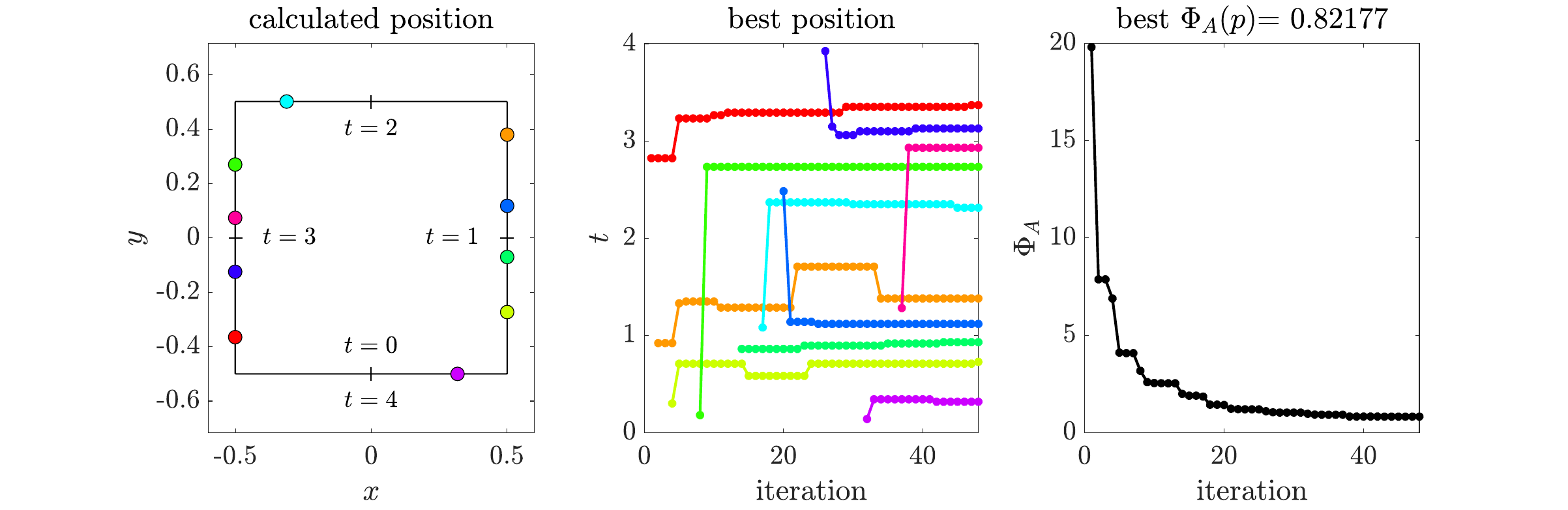}
\caption{Sequential algorithm enhanced by gradient descent. Left: Optimized design for ten pressure activations. Middle: Progress of the algorithm, with an ``iteration'' referring to the introduction of a new pressure activation via a one-dimensional exhaustive search or a step of gradient descent for fine-tuning the design after such an introduction. Right: Evolution of the optimization target, with the final value $\Phi_{\rm A}(p^*)=0.82$.
}\label{10_acti_20_sensor_with_line_search}
\end{figure}

Figure~\ref{10_acti_20_sensor_with_line_search}, which is organized in the same way as Figure~\ref{sequal_without_line_search}, visualizes the progress of the enhanced sequential search that applies gradient descent to the whole design each time after introducing a new pressure activation. Although the fine-tuning by gradient descent seems to significantly affect the structure of the design at some intermediate steps of the algorithm, the final optimized configuration with ten pressure activations is qualitatively similar to that obtained with the basic sequential search. In this case, the optimized design corresponds to the A-optimality target value $\Phi_{\rm A}(p^*) = 0.82$ that is about $6$\% less than without enhancing the algorithm with intermediate gradient descent steps.

A direct application of gradient descent to an initial guess of an approximately equidistant set of pressure activations is documented in Figure~\ref{sequal_with_line_search_it_max_15}. As always, the left-hand image shows the optimized activation positions, the middle image visualizes how they move, and the right-hand image depicts the evolution of the A-optimality target. The optimized design is qualitatively similar to those produced by the sequential algorithms, with the final value of $\Phi_{\rm A}(p^*) = 0.82$ which is (almost) the same as for the enhanced sequential search. For this particular setup, the enhanced sequential algorithm and mere  gradient descent were thus able to find equally good local minimizers. However, it should be noted that the quality of the local minimizer produced by gradient descent often depends on the initial guess.

\begin{figure}[t]
\centering 
\includegraphics[width=1\textwidth]{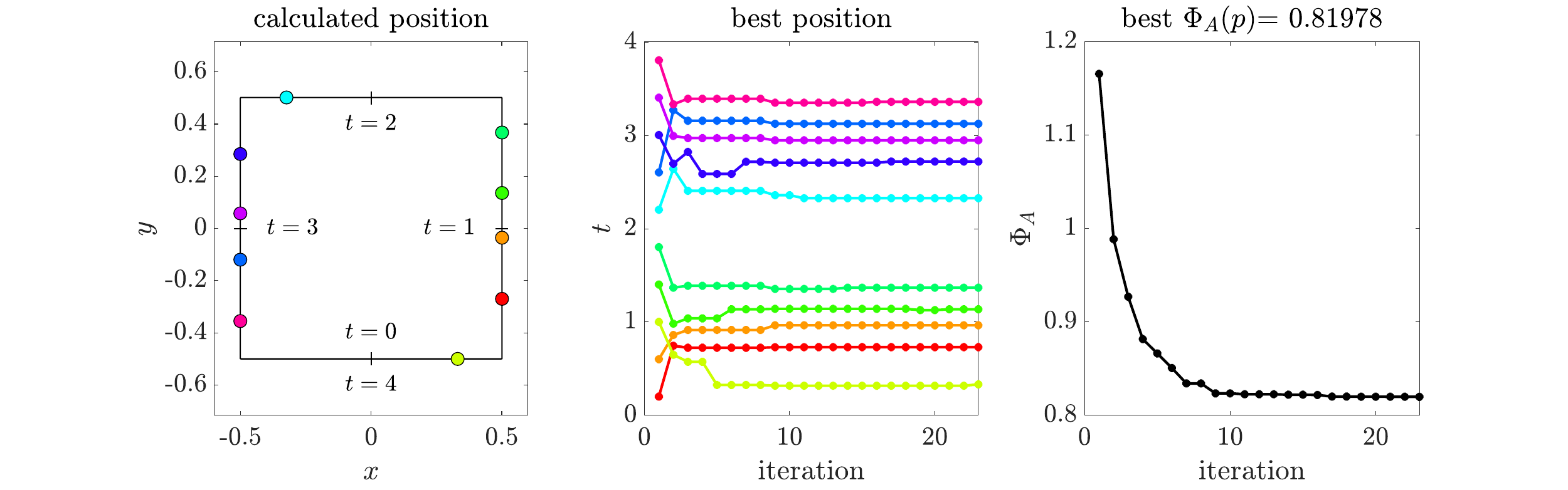}
\caption{Gradient descent with an equidistant initial guess for the pressure activations. Left: Optimized design of ten pressure activations. Middle: Progress of the algorithm. Right: Evolution of the optimization target, with the final value $\Phi_{\rm A}(p^*)=0.82$.
}\label{sequal_with_line_search_it_max_15}
\end{figure}

\subsection{Designs for stiff and loose background Lam\'e parameters}

As the final experiment, we repeat for ``stiff'' and ``loose'' background materials the exhaustive search with three pressure activations presented in Section~\ref{sec:num1}. The background Lam\'e parameters chosen for the stiff material are $\lambda_0=3.9188\cdot 10^{10}$\,Pa and $\mu_0= 5.3675\cdot 10^{10}$\,Pa that correspond to gray cast iron, whereas those employed for the loose material,~i.e.,~$\lambda_0=8.1081\cdot 10^8$\,Pa and $\mu_0=3.3784\cdot 10^7$\,Pa, could model rubber.

The results for the stiff and loose background materials are, respectively, visualized in Figures~\ref{stiff} and \ref{loose} that have the same structure as Figure~\ref{exhaustive} for a material of intermediate stiffness (acrylplastic) with $\lambda_0=2.7654\cdot 10^9$\,Pa and $\mu_0=1.1852 \cdot 10^9$\,Pa. By comparing the three figures, it becomes apparent that the optimal design depends on the properties of the background material, but it is difficult to give intuitive explanations for the differences between the deduced designs. Moreover, it is probable that, in addition to the stiffness/looseness of the material, also the ratio between $\mu_0$ and $\lambda_0$ affects the optimal design: recall that in Section~\ref{sec:priors}, it was assumed that the perturbations in $\lambda$ and $\mu$ have the same pointwise standard deviations, which means that the assumed level of relative change in the two Lam\'e parameters depends on the respective background values.

\begin{figure}[t]
\centering 
\includegraphics[width=1\textwidth]{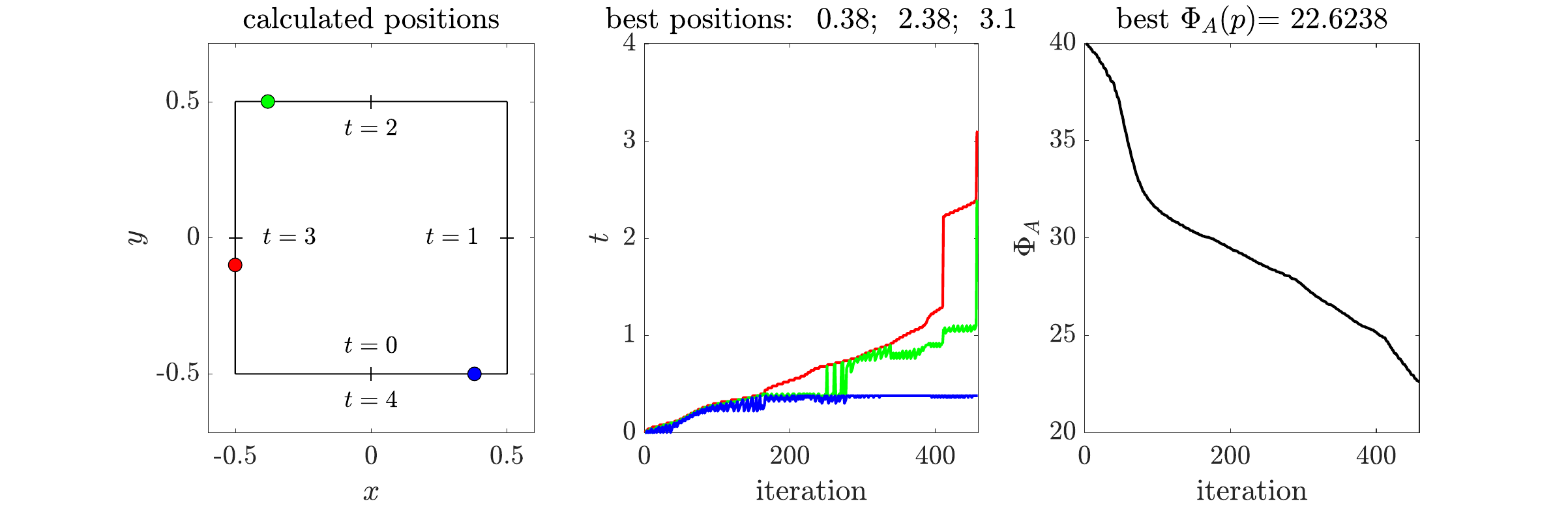}
\caption{Exhaustive search for a {\em stiff} material (gray cast iron). Left: Optimal design of three pressure activations corresponding to the arclength parameter triplet $p^* = (0.38, 2.38, 3.1)$. Middle: Progress of the search, with an ``iteration'' referring to the instances when the estimate for $p^*$ was updated in the exhaustive search. Right: Evolution of the optimization target, with the final optimal value $\Phi_{\rm A}(p^*)=22.62$.
}\label{stiff}
\end{figure}

The optimal designs for the stiff and loose materials in Figures~\ref{stiff} and \ref{loose} correspond to the A-optimality target values of $22.62$ and $0.12$, respectively. If one adopts the optimal design deduced for the intermediately stiff material from Figure~\ref{exhaustive} for the stiff and loose cases, the resulting values for $\Phi_{\rm A}$ are $25.75$ and $0.28$, which are considerably worse than the optimal values for these cases. This further demonstrates the effect that the background Lam\'e parameters have on A-optimality of experimental designs. According to Figures~\ref{stiff}, \ref{exhaustive} and \ref{loose}, the looser the background material is, the lower is the value of $\Phi_{\rm A}$ that measures the expected squared reconstruction error for the linearized model.

\begin{figure}[t]
\centering 
\includegraphics[width=1\textwidth]{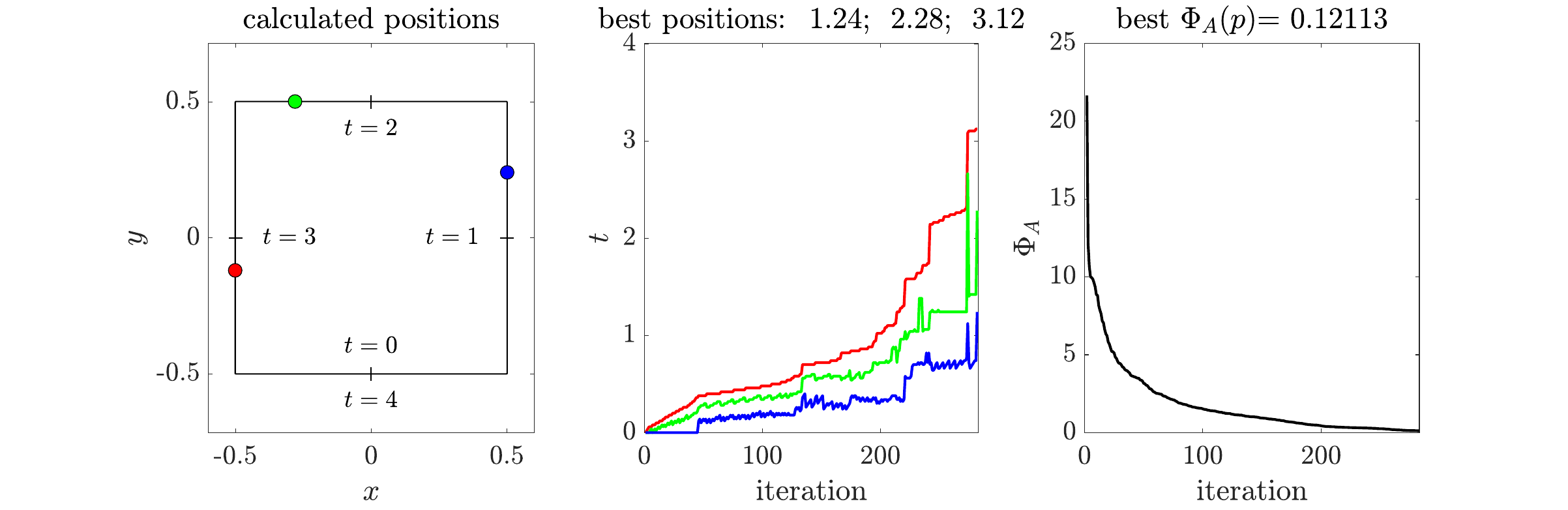}
\caption{Exhaustive search for {\em soft} material (rubber). Left: Optimal design of three pressure activations corresponding to the arclength parameter triplet $p^* = (1.24, 2.28, 3.12)$. Middle: Progress of the search, with an ``iteration'' referring to the instances when the estimate for $p^*$ was updated in the exhaustive search. Right: Evolution of the optimization target, with the final optimal value $\Phi_{\rm A}(p^*)=0.12$.
}\label{loose}
\end{figure}
\noindent

\section{Concluding remarks}
\label{sec:conclusions}

This work introduced Bayesian OED as a design tool for the two-dimensional inverse boundary value problem of linear elasticity with a realistic discrete model for the boundary measurements. Based on a linearization of the forward model around a background level for the Lam\'e parameters, a couple of algorithms for finding A-optimal experimental designs with respect to the positions of the applied boundary pressure activations were introduced and numerically tested. All considered optimization algorithms were able to significantly decrease the value of the A-optimality target function compared to,~e.g.,~random designs, but apart from an exhaustive search none of them could reliably find the global optimum. Indeed, based on our numerical experiments, the optimization process typically suffers from a number local minima, and coping with them is arguably an important aspect to be considered in further studies on this topic. Other interesting topics for future investigations are understanding how the scales for the two Lam\'e parameters $\lambda$ and $\mu$ or a ROI inside the imaged object affect optimal designs as well as applying Bayesian OED to a three-dimensional setting for linear elasticity or to its original nonlinear inverse boundary value problem.

\appendix

\section{Parametrization of the domain}
\label{app:A}

For the sake of completeness, this appendix gives an arclength parametrization $\gamma: [0, L) \to \partial \Omega \subset \R^2$ for the boundary of the computational domain $\Omega$. Recall that $\Omega$  is the rounded square with circumference $L= 2\pi r+4$ shown in Figure~\ref{setting_test_object}, with $r$ being the radius of the circular arcs at the corners and the parameter value $t=0$ corresponding to the midpoint of the bottom face of the square. The first component of $\gamma$ is given as
\begin{align*}
\gamma_1 (t)&=\left\{
\begin{array}{ll}
		t, \quad &0\leq t <0.5, \\[1mm]
    0.5+r\cos\left(\frac{t-0.5}{r}+\frac32\pi\right), \quad &0.5\leq t<\frac{\pi}{2}r+0.5,\\[1mm]
    r+0.5, \quad &\frac{\pi}{2}r+0.5\leq t<\frac{\pi}{2}r+1.5, \\
		0.5+r\cos\left(\frac{t-\left(\frac{\pi}{2}r+1.5\right)}{r}\right), \quad &\frac{\pi}{2}r+1.5\leq t<\pi r+1.5, \\[1mm]
    \pi r+2-t, \quad &\pi r+1.5\leq t<\pi r+2.5, \\[1mm]
    r\cos\left(\frac{t-(\pi r+2.5)}{r}+\frac{\pi}{2}\right)-0.5, \quad &\pi r+2.5\leq t<\frac32\pi r+2.5, \\[1.5mm]
    -r-0.5, \quad &\frac32\pi r+2.5\leq t<\frac32\pi r+3.5,\\    	
		r\cos\left(\frac{t-\left(\frac32\pi r+3.5\right)}{r}+\pi\right)-0.5, \quad &\frac32\pi r+3.5\leq t<2\pi r+3.5,\\[1mm]
    t- (2\pi r + 4), \quad &2\pi r+3.5\leq t<2\pi r+4.
\end{array}
\right.
\end{align*}
and the second component as
\begin{align*}
\gamma_2(t)&=\left\{
\begin{array}{ll}
		-0.5-r, \quad &0\leq t<0.5,\\[1mm]
    -0.5+r\sin\left(\frac{t-0.5}{r}+\frac32\pi\right), \quad &0.5\leq t<\frac{\pi}{2}r+0.5,\\[1mm]
    -0.5+t-\left(\frac{\pi}{2}r+0.5\right), \quad &\frac{\pi}{2}r+0.5\leq t<\frac{\pi}{2}r+1.5,\\[1mm]
		0.5+r\sin\left(\frac{t-\left(\frac{\pi}{2}r+1.5\right)}{r}\right), \quad &\frac{\pi}{2}r+1.5\leq t<\pi r+1.5,\\
    0.5+r, \quad &\pi r+1.5\leq t<\pi r+2.5,\\[1mm]
    0.5+r\sin\left(\frac{t-(\pi r+2.5)}{r}+\frac{\pi}{2}\right), \quad &\pi r+2.5\leq t<\frac32\pi r+2.5,\\[1mm]
    \frac32\pi r+3-t, \quad &\frac32\pi r+2.5\leq t<\frac32\pi r+3.5,\\   	
		-0.5+r\sin\left(\frac{t-\left(\frac32\pi r+3.5\right)}{r}+\pi\right), \quad &\frac32\pi r+3.5\leq t<2\pi r+3.5,\\[1mm]
    -(0.5+r), \quad &2\pi r+3.5\leq t<2\pi r+4.
\end{array}
\right.
\end{align*}
\noindent
We want to remark that $\gamma(t)$ is continuously differentiable and invertible, with its inverse $\gamma^{-1}: \partial \Omega \to [0,L)$ given by
\begin{align*}
\gamma^{-1}(x,y)&=\left\{\begin{array}{ll}
x, \quad & 0\leq x \leq 0.5, \ \  y = -0.5-r,\\[1mm]
    0.5+r\left(\arccos\left(-\frac{x-0.5}{r}\right)-\frac{\pi}{2}\right), \quad &0.5<x<0.5+r, \ \ -0.5-r < y < -0.5,\\[1mm]
    1+r\frac{\pi}{2}+y, \quad & x = 0.5 + r, \ \  -0.5 \leq y \leq 0.5,   \\[1mm]
    1.5+r\frac{\pi}{2}+r\arccos\left(\frac{x-0.5}{r}\right), \quad &0.5<x<0.5+r,\ \ 0.5< y < 0.5 + r,\\[1mm]
    1.5+r\pi-(x-0.5), \quad & - 0.5 \leq  x  \leq 0.5, \ \ y = 0.5 +r, \\[1mm]
    2.5+r\frac{\pi}{2}+r\arccos\left(\frac{x+0.5}{r}\right), \quad &-0.5 -r < x<-0.5, \ \  0.5 <y < 0.5 + r, \\[1mm]    
 2.5+3r\frac{\pi}{2}-(y-0.5), \quad &x= -0.5 - r, \ \ -0.5 \leq y \leq 0.5,\\[1mm]
    3.5+3r\frac{\pi}{2}+r\pi -r\arccos\left(\frac{x+0.5}{r}\right), \quad & -0.5-r < x < -0.5, \ \  -0.5-r < y <-0.5,\\[1mm]
    4+2\pi r+x,  \quad & -0.5 \leq  x \leq 0, \ \  y =  -0.5-r,
		\end{array}
\right.
\end{align*}
where $\arccos$ denotes the principal value of arcus cosine.

\bibliographystyle{acm}
\bibliography{measurement_optimization}
\end{document}